\documentclass[12pt]{article}
\usepackage{a4wide,amsmath,amssymb,amsthm,graphicx,color,comment}
\newtheorem{theorem}{Theorem}[section]
\newtheorem{corollary}[theorem]{Corollary}

\newtheorem{lemma}[theorem]{Lemma}
\newtheorem{remark}[theorem]{Remark}
\newcommand{\uld}[1]{\underline{d#1}}

\newcommand\nlc{\eta}
\newcommand\uvec{\vec{u}}
\newcommand\pvec{\vec{p}}
\newcommand\rvec{\vec{r}}

\newcommand\N{\mathbb{N}}
\newcommand\XXX{\mathbb{X}}
\newcommand\YYY{\mathbb{Y}}
\newcommand\pole{\mathfrak{p}}
\newcommand\argu{\theta}
\newcommand\radi{r}
\newcommand\omzer{\omega_0}
\renewcommand\Re{\mathrm{Re}}
\renewcommand\Im{\mathrm{Im}}
\definecolor{darkcyan}{rgb}{0.,0.5,0.5}
\definecolor{darkviolet}{rgb}{0.5,0.,0.5}
\definecolor{darkgreen}{rgb}{0.,0.7,0.}
\newcommand{\commBK}[1]{}
\newcommand{\skipKuznetsov}[1]{}
\newcommand{\revision}[1]{{#1}}
\newcommand{\marginrr}[1]{}
\newcommand{\marginR}[1]{}

\title{Acoustic nonlinearity parameter tomography with the Jordan-Moore-Gibson-Thompson equation in frequency domain}
\author{Barbara Kaltenbacher, University of Klagenfurt
(barbara.kaltenbacher@aau.at)
}
\begin{document}
\maketitle

\begin{abstract}
This paper aims to combine the advantages of the Jordan-Moore-Gibson-Thompson JMGT equation as an advanced model in nonlinear acoustics with a frequency domain formulation of the forward and inverse problem of acoustic nonlinearity parameter tomography, enabling the multiplication of information by nonlinearity.
Our main result is local uniqueness of the space dependent nonlinearity parameter from boundary measurements, which we achieve by linearized uniqueness with an Implicit Function type perturbation argument in appropriately chosen topologies.
Moreover, we shortly dwell on the application of a regularized Newton type method for reconstructing the nonlinearity coefficient, whose convergence can be established by means of the linear uniqueness result. 
\end{abstract}



\section{Introduction}
Exploiting the nonlinearity coefficient as an imaging quantity in ultrasonics, that is, acoustic nonlinearity parameter tomography ANT~\cite{nonlinparam1, duck2002nonlinear,  nonlinparam3, nonlinparam2, ZHANG20011359}
has been discovered as a promising alternative or actually complement to ultrasound tomography UST. As such it is another way of going beyond conventional ultrasonography and obtaining quantitative images by exploiting tissue dependence of the nonlinearity coefficient (while UST relies on tissue dependence of the sound speed).
While the above cited results are mostly of experimental nature, the mathematical and numerical analysis of models of nonlinear acoustics has only recently enabled and started off mathematical research on this topic, that is expected to provide computational tools for this innovative imaging modality  
~\cite{AcostaUhlmannZhai2021,Eptaminitakis:Stefanov:2023,nonlinearity_imaging_JMGT,nonlinearity_imaging_Westervelt,nonlinearity_imaging_fracWest,nonlinearity_imaging_2d,nonlinearityimaging}.
Most of these results rely on the Westervelt equation,  
a quasilinear second order wave equation, to describe nonlinear ultrasound propagation.
However, it is well-known that this model has its limitations (in particular, it exhibits an infinite speed of propagation paradox) and therefore, advanced models of nonlinear acoustics need to be employed, also for the imaging task of ANT. 
On the other hand, considering the forward and inverse problem of ANT in frequency domain is required in the practically relevant case of periodic excitations with a source operating at a single or a few frequencies and has turned out to be highly beneficial for the inverse problem of ANT, where the multiplication of information due to nonlinearity becomes clearly visible in frequency domain formulations \cite{nonlinearity_imaging_2d}. The effect that ``nonlinearity helps'' \cite{KurylevLassasUhlmann2018} by the way not only affects the recovery of the nonlinearity coefficient in ANT but also the inverse problem of reconstructing the sound speed in UST, as has been demonstrated computationally in \cite{kWave_nonlinear} and by generation of higher harmonics is in fact already routinely used in ultrasonography to reduce artifacts \cite{burns2000nonlinear}. 

As an advanced model for nonlinear ultrasound propagation, the Jordan-Moore-Gibson-Thompson JMGT  equation~\cite{JordanMaxwellCattaneo09,JordanMaxwellCattaneo14} 
\begin{equation}\label{JMGT-Westervelt}
\tau u_{ttt}+u_{tt}+\omzer u_t-c^2\Delta u -b \Delta u_t + \nlc (u^2)_{tt} +f=0,
\end{equation}
has recently attracted much interest and initial-boundary value problems for this PDE have been intensively studied, along with their linearizations, in the recent literature; some selected results on well-posedness, regularity of solutions, and long-term behavior of initial value problems for these equations can be found in, 
e.g.,~\cite{BongartiCharoenphonLasiecka20,bucci2020regularity,chen2019nonexistence,DellOroPata,KLM12_MooreGibson,KLP12_JordanMooreGibson,JMGT,JMGT_Neumann,LiuTriggiani13,MarchandMcDevittTriggiani12,NikolicSaidHouari21a,NikolicSaidHouari21b,PellicerSolaMorales,racke2020global}.

In \eqref{JMGT-Westervelt}, $\tau>0$ is the relaxation time, $c>0$ the speed of sound, $b=\tau c^2+\delta>0$ where $\delta\geq 0$ is the diffusivity of sound, $\omzer\geq0$ is a weak attenuation coefficient, 
and equation \eqref{JMGT-Westervelt} is supposed to hold on some space-time domain $(0,T)\times\Omega$ with $\Omega\in\mathbb{R}^d$, $d\in\{1,2,3\}$.
In case $\tau=0$ \eqref{JMGT-Westervelt} becomes the classical Westervelt model, a second order strongly damped quasilinear wave equation.

The inverse problem of ANT with the JMGT equation as a model for ultrasound propagation has been studied in \cite{nonlinearity_imaging_JMGT} in time domain, where it turned out that indeed a positive relaxation time $\tau>0$ allows to prove local uniqueness for the inverse problem, whereas in the Westervelt case only linearized uniqueness could be obtained. This is why in this paper we continue research in this direction and further explore the potential of ANT with the JMGT equation by studying it in frequency domain.

For the forward problem, we can rely on the results from \cite{periodicJMGT}, where in place of an initial value problem (as underlying all the above cited results), we postulate temporal periodicity 
\begin{equation}\label{periodic}
u(T)=u(0), \quad u_t(T)=u_t(0)
, \quad \tau u_{tt}(T)=\tau u_{tt}(0)
\end{equation}
for some $T>0$; 
this is practically 
\revision{relevant}
\marginrr{3.p2}
under sinosoidal continuous wave excitations.

On the boundary $\partial\Omega=\Gamma_a\cup\Gamma_i\cup\Gamma_N\cup\Gamma_D$ of the spatial domain where \eqref{JMGT-Westervelt} is supposed to hold, in order to capture a wide range of practically relevant scenarios, we impose  
mixed boundary conditions 
\begin{equation}\label{bndy}
\begin{aligned}
&\partial_\nu u+\beta u_t+\gamma u =0 \mbox{ on }\Gamma_a\cup\Gamma_i\cup\Gamma_N\\
&u =0 \mbox{ on }\Gamma_D\\
&\text{ where $\gamma=0$ on $\Gamma_N$, $\tfrac{1}{\gamma}\vert_{\Gamma_i}\in L^\infty(\Gamma_i)$, $\beta=0$ on $\Gamma_i\cup\Gamma_N$,}
\end{aligned}
\end{equation}
and consider homogeneous boundary data for simplicity \footnote{pointing to, e.g.,  \cite{FSI} for a standard homogenization approach, applied to a nonlinear wave equation to deal with inhomogeneous boundary conditions}.
This corresponds to 
\underline{a}bsorbing boundary conditions on $\Gamma_a$ to avoid spurious reflections and mimic open domain wave propagation; as well as 
\underline{i}mpedance, \underline{N}eumann, and \underline{D}irichlet boundary conditions on $\Gamma_i$, $\Gamma_N$, and $\Gamma_D$, to model damping, sound-hard and sound-soft boundary parts, respectively. 

\medskip
The inverse problem of nonlinearity parameter imaging amounts to reconstructing the space-dependent coefficients $\nlc=\nlc(x)$ in \eqref{JMGT-Westervelt}, \eqref{periodic}, \eqref{bndy} from additional measurements of the acoustic pressure $p^{obs}$ taken at a receiver array, which can be written as the  Dirichlet trace on some manifold $\Sigma$ immersed in the domain $\Omega$ or attached to its boundary $\Sigma\subseteq\overline{\Omega}$
\begin{equation}\label{observation}
p^{obs}(t,x) = u(t,x), \quad(t,x)\in(0,T)\times \Sigma.
\end{equation}

Uniqueness for this inverse problem with periodicity condition \eqref{periodic} replaced by initial conditions has been studied in \cite{nonlinearity_imaging_JMGT}.
\medskip

The goals of this work are to combine 
\begin{itemize}
\item the advantages of the \textit{JMGT equation} (as compared to Westervelt) as an advanced model from the physical point of view that also mathematically (in view of its hyperbolic rather than parabolic nature) enhances the chances to prove uniqueness, with 
\item the benefits of a \textit{formulation in frequency domain}, that is not only practically relevant for continuous wave  excitation, but also allows to quantify the \revision{multiplication} 
\marginR{1}
of information due to nonlinearity. 
\end{itemize}
Indeed, 
as compared to \cite{AcostaUhlmannZhai2021}, we are here able to prove (local) uniqueness from a single excitation; 
as compared to  \cite{nonlinearity_imaging_Westervelt,nonlinearity_imaging_2d}, it is fully nonlinear rather than just linearized uniqueness and
as compared to \cite{nonlinearity_imaging_JMGT}, we can even quantify uniqueness as a stability estimate in properly chosen (Sobolev type) norms on $\nlc$ and $u$. Also the topology in which we define 
\revision{a neighborhood} 
\marginR{2}
\marginrr{3.p3}
in local uniqueness is of Sobolev type. 
Additionally to this, we also provide some results on Newton's method for the underlying inverse problem and some numerical illustration of the dependence of the degree of ill-posedness of the inverse problem on the magnitude of the relaxation time.

The structure of the paper is as follows:
\begin{itemize}
\item In Section~\ref{sec:multiharmonic}, we provide a multiharmonic expansion and formulate the inverse problem in frequency domain;
\item The main results of this paper are contained in Section~\ref{sec:uniqueness}, where we prove local uniqueness of $(\nlc,u)$ around a properly chosen reference point $(\nlc^0,u^0)$.
More precisely, we choose $\nlc^0=0$, and $u^0$ to be space-frequency separable;
\item The brief final 
Section~\ref{sec:Newton} is devoted to a study of Newton's method for this inverse problem in 
an all-at-once setting (considering model and observation equations as a simultaneous system for coefficient and state $(\nlc,u)$).
\end{itemize}

\subsection*{Notation 
}
\paragraph{Indices and vector notation:}
By $\vec{\cdot}$ we denote infinite vectors $({\cdot}_m)_{m\in\N}$ of temporal frequencies.
We use indices $m$ for temporal and $j$ or $\ell$ for spatial frequencies (corresponding to eigenvalues of the negative Laplacian)
\paragraph{Function spaces:}
We will frequently use Bochner spaces and sometimes make use of the abbreviations
$L^p(Z):=L^p(0,T;Z)$, $H^\sigma(Z):=H^\sigma(0,T;Z)$ for some space $Z$ of space dependent functions. 
In frequency domain we will use their counterparts $h^\sigma(Z)$ equipped with the norm
$\|\vec{v}\|_{h^\sigma(Z)}=\left(\sum_{m\in\mathbb{N}} (m\omega)^{2\sigma} \|v_m\|_Z^2\right)^{1/2}$.

\section{Multiharmonic formulation}\label{sec:multiharmonic}

The inverse problem in time domain reads as 
\[
\begin{aligned}
&\text{model:}&&{}
[\tau\mathcal{M}_0\partial_t^3+\mathcal{M}_\beta\partial_t^2 +\omzer\mathcal{M}_0\partial_t + c^2\mathcal{A} + (\tau c^2+\delta) \mathcal{D}\partial_t]u + \nlc\, (u^2)_{tt} = r_{tt}\\
&&&{} \int_0^T u(t)\, dt =0, \quad u(T)=u(0), \quad u_t(T)=u_t(0)
\\
&\text{observation:}&& u(x_0,t)=p^{obs}(x_0) \quad x_0\in\Sigma
\end{aligned}
\]
where $\mathcal{A}$, $\mathcal{D}$ denote the negative Laplacian and $\mathcal{M}_\beta$ a modified (in case $\beta\not=0$)  version of the identity, equipped with boundary conditions according to \eqref{bndy} 
\begin{equation}\label{eqn:calADM}
\begin{aligned}
&\mathcal{A}u = \Bigl(v\mapsto \int_\Omega\nabla u\cdot\nabla v\, dx +\gamma\int_{\partial\Omega} u\, v\, ds\Bigr)\\
&\mathcal{D}u = \Bigl(v\mapsto \int_\Omega\nabla u\cdot\nabla v\, dx +(\tfrac{c^2\beta}{\tau c^2+\delta}+\gamma)\int_{\partial\Omega} u\, v\, ds\Bigr)\,, \\
&\mathcal{M}_\beta u = \Bigl(v\mapsto \int_\Omega u\,v\, dx +\beta (\tau c^2+\delta)\int_{\partial\Omega} u\, v\, ds\Bigr).
\end{aligned}
\end{equation}

Note that, as observed in \cite{periodicWest_2}, the excitation should be the second time derivative of a pressure source in order to equip it with sufficient strength in terms of the order in which it appears in the PDE; this also corresponds to a matching of physical units.

In case of a $T$-periodic source 
$r(x,t)=\Re\left(\sum_{k=1}^\infty \hat{r}_k(x) e^{\imath k \omega t}\right)$ (with $\omega=\frac{2\pi}{T}$), motivated by  existence of a $T$-periodic solution \cite{periodicJMGT},
we can expand $u$ with respect to the basis $e^{\imath k \omega t}$ as well 
\[
u(x,t)
= \Re\left(\sum_{k=1}^\infty \hat{u}_k(x) e^{\imath k \omega t}\right),
\]
which analogously to the Westervelt case $\tau=0$ \cite{periodicWestervelt,periodicWest_2} leads to the multiharmonic system
\begin{subequations}\label{IPfreq}
\begin{alignat}{2}
&
F_m^{mod}(\nlc,\uvec):=
\mathcal{L}_m\hat{u}_m +\nlc\, \mathcal{B}_m(\uvec,\uvec) &&= \hat{r}_m
\label{IPfreq_mod}\\
&
F_m^{obs}(\nlc,\uvec):= \text{tr}_\Sigma\hat{u}_m &&=\hat{p}^{obs}_m,
\label{IPfreq_obs}
\end{alignat}
\end{subequations}
for all $m\in\mathbb{N}$, where
\begin{equation}\label{Bm}
\begin{aligned}
&\uvec=(\hat{u}_m)_{m\in\mathbb{N}}, \quad \hat{u}_m(x)=\revision{\frac{2}{T}}\int_0^T  u(x,t)\,e^{-\imath m \omega t}\, dt, \ x\in\Omega,\\
&\mathcal{L}_m = \frac{1}{m^2\omega^2}
[\imath m^3\omega^3\tau\mathcal{M}_0+m^2\omega^2\mathcal{M}_\beta
\revision{-}\imath m\omega\omzer\mathcal{M}_0 \revision{-} c^2\mathcal{A} \revision{-} (\tau c^2+\delta) \imath m\omega \mathcal{D}]\\
&\mathcal{B}_m(\uvec,\vec{v}) = 
\revision{
\frac12 \sum_{\ell=1}^{m-1} \hat{u}_\ell \hat{v}_{m-\ell}
+\frac12 \sum_{k=1}^\infty\overline{\hat{u}_{k}} \hat{v}_{k+m} 
+\frac12 \sum_{k=1}^\infty\hat{u}_{m+k} \overline{\hat{v}_{k}}
}
=\revision{\frac{2}{T}}\int_0^T  \revision{u(\cdot,t)\,v(\cdot,t)}\,e^{-\imath m \omega t}\, dt,\\
&\vec{p}^{obs}_m=(\hat{p}^{obs}_m)_{m\in\mathbb{N}}, \quad 
\hat{p}^{obs}_m(x_0)=\revision{\frac{2}{T}}\int_0^T  p^{obs}(x_0,t)\,e^{-\imath m \omega t}\, dt , \ x_0\in\Sigma.
\end{aligned}
\end{equation}
\marginrr{2.1.1}
\marginrr{2.1.2}
\marginrr{3.p5-1}
Note that the coefficient $\nlc$ appears in each of the 
\revision{infinitely} 
\marginrr{3.p5-2}
many equations \eqref{IPfreq}, with factors $\mathcal{B}_m(\uvec,\uvec)$ being determined by the pressure, which reflects the physical effect of wave interaction generating interior sources. The fact that --- as opposed to a single 
\revision{Helmholtz} 
\marginrr{3.p5-3}
type equation in the linear case --- we here end up with an infinite system at multiples of the excitation frequency is known as generation of higher harmonics in the physics literature on nonlinear acoustics.

\medskip

With
\[
F_m:= \left(\begin{array}{l}
F_m^{mod} \\
F_m^{obs}
\end{array}\right), \qquad
y_m:= \left(\begin{array}{l}
\hat{r}_m \\
\hat{p}^{obs}_m
\end{array}\right)
\]
this can be written as
\begin{equation}\label{Fnlcuvecy}
\vec{F}(\nlc,\uvec) = \vec{y}.
\end{equation}
This is an all-at-once formulation of the inverse problem, which considers the coefficient $\nlc$ and the state $\uvec$ as simultaneous variables.
Note that $\mathcal{B}_m$ is not linear with respect to the second component (since it contains/misses some complex conjugations) but still additively linear in the sense that 
$\mathcal{B}_m(\uvec,\vec{v}_1+\vec{v}_2)=\mathcal{B}_m(\uvec,\vec{v}_1)+\mathcal{B}_m(\uvec,\vec{v}_2)$.

The operator equation \eqref{Fnlcuvecy} will be considered in two different function space settings in this paper.
In Section~\ref{sec:Newton}, where we study Newton's method for the regularized solution of this inverse problem, we use a Sobolev and Lebesgue space setting that reflects a realistic coefficient, state and data space scenario.
The uniqueness analysis in Section~\ref{sec:uniqueness} relies on a construction of function spaces that allows \revision{for the} 
\marginR{3}
\marginrr{3.p5-4}
application of an Inverse Function Theorem type argument. Boundedness of the linearization of $F$ at a reference point on the one hand and smoothness of $F$ on the other hand are competing requirements whose simultaneous fulfillment necessitates an appropriate tuning of space definitions. 
Note that uniqueness per se is independent of the choice of topology in data space, so only the strength of the regularity assumptions imposed on the coefficient counts, which is actually moderate, namely Sobolev smoothness of order \revision{$s\in(\frac12,1]$}, also in the locality assumption for nonlinear uniqueness.
The Lipschitz stability estimates we will establish in this paper are only a side product of the uniqueness proof and actually conform to severe ill-posedness of the problem as they amount to homomorphy under an (artificially) strong data space norm.

\section{Uniqueness and stability}\label{sec:uniqueness}
We first of all provide a linearized uniqueness and stability result (Section~\ref{sec:uniqueness-lin}), which together with a perturbation argument will be used to prove local uniqueness of the fully nonlinear inverse problem (Section~\ref{sec:uniqueness-nonlin}).

For simplicity of exposition we will consider $\beta=0$, so that $\mathcal{D}=\mathcal{A}$ and $\mathcal{M}=\textrm{id}$ here and make use of an eigensystem $(\lambda^j,(\varphi^{j,k})_{k\in K^j})_{j\in\N}$ of $\mathcal{A}$ as well as 
\revision{
the eigenspaces $\mathbb{E}^j=\text{span}\{\varphi^{j,k}\,:\, k\in K^j\}$  corresponding to the eigenvalues $\lambda^j$ of $\mathcal{A}$ and }
\marginrr{2.2.2}
the induced norm on $H^s(\Omega)$
\begin{equation}\label{eigensystem}
\mathcal{A}\varphi^{j,k}=\lambda^j\varphi^{j,k}, \quad \langle \varphi^{j,k},\varphi^{i,\ell}\rangle=\delta_{ji}\delta_{k\ell}, \quad \|v\|_{H^s(\Omega)}=\left(\sum_{j\in\N} (\lambda^j)^s\sum_{k\in K^j} |\langle \varphi^{j,k},v\rangle|^2\right)^{1/2}.
\end{equation}
 
\begin{remark}\label{rem:generalmathcalD}
The more general setting of $\mathcal{D}$, $\mathcal{A}$ being jointly diagonalizable can be tackled by similar techniques (see \cite[Theorem 3.2]{nonlinearity_imaging_2d} for a linearized uniqueness proof under such conditions). In particular, this comprises space-fractional damping $\mathcal{D}=\mathcal{A}^\alpha$ with $\alpha\in\mathbb{R}^+$.
On the other hand, for a proof of linearized uniqueness with time fractional damping, that is, replacing $\partial_t$ by a fractional derivative $\partial_t^\alpha$, we refer to results in \cite{nonlinearity_imaging_fracWest} on the Westervelt equation, that we expect to be adoptable to the JMGT setting.
Replacing initial by periodicity conditions in a time fractional setting poses many open questions in view of the fact that spaces of exponential or trigonometric functions are not closed under application of, e.g, the Riemann-Liouville derivative.
\end{remark}

\subsection{
Continuous invertibility of linearized all-at-once forward operator} \label{sec:uniqueness-lin}
We linearize $F$ at a particularly chosen point $(\nlc,\uvec)$, namely 
\begin{equation}\label{u0separable}
\nlc=0, \quad \hat{u}^0_m(x)=\phi(x) \psi_m  \ \text{ with }\phi\not=0\text{ a.e. in }\Omega, \quad \mathcal{B}_m(\vec{\psi},\vec{\psi})\not=0, \ m\in\N.
\end{equation}
It can easily be verified that by a proper choice of $\vec{r}$, existence of such a space-frequency separable solution to $F^{mod}(0,\uvec^0)=\vec{r}$ can be established.
(However, note that in an all-at-once formulation, the state part of the unknowns does not necessarily need to be a PDE solution corresponding to the coefficient part of the unknowns, anyway.) 

The goal of this section is thus to prove that $\vec{F}'(0,\uvec^0):\XXX\to\YYY$ is injective with bounded inverse between appropriately chosen normed spaces $\XXX$ and $\YYY$.
This amounts to estimating the $\XXX$ norm of $(\uld{\nlc},\uld{\uvec})$ by a constant multiple of the $\YYY$ norm of $(\uld{\rvec},\uld{\pvec}^{obs})$ in 
\[
(\uld{\rvec},\uld{\pvec}^{obs})=\vec{F}'(0,\uvec^0)[(\uld{\nlc},\uld{\uvec})].
\]
\revision{
The linarization of $\vec{F}$ is given by 
\begin{subequations}\label{IPfreq_diff}
\begin{alignat}{2}
&{F_m^{mod}}'(\nlc,\uvec)[\uld{\nlc},\uld{\uvec}]=
\mathcal{L}_m\uld{u}_m 
+\uld{\nlc}\, \mathcal{B}_m(\uvec,\uvec)+\nlc\Bigl( \mathcal{B}_m(\uld{\uvec},\uvec)+\mathcal{B}_m(\uvec,\uld{\uvec})\Bigr)
\label{IPfreq_diff_mod}\\
&{F_m^{obs}}'(\nlc,\uvec)[\uld{\nlc},\uld{\uvec}]= \text{tr}_\Sigma\uld{u}_m
\label{IPfreq_diff_obs}
\end{alignat}
\end{subequations}
For $\nlc\not=0$, this is only formal, since due to the complex conjugations (not) present in some of the terms in $\mathcal{B}_m$, it is not a bilinear operator, which inhibits complex differentiability. 
Note however, that we only apply linearization at $\nlc=0$, where indeed it defines a Fr\'{e}chet derivative. 
}
\marginrr{3.p6-1}
We project \eqref{IPfreq_diff_mod} onto the basis elements $\varphi^{j,k}$;  
or equivalently, project the JMGT equation,  linearized at $\nlc=0$, 
\begin{equation}\label{JMGT-lin}
\tau \uld{u}_{ttt}+\uld{u}_{tt}+\omzer\uld{u}_t+c^2\mathcal{A} \uld{u} +(\tau c^2+\delta) \mathcal{A} \uld{u}_t +\uld{\nlc}\,\phi^2(x)\revision{(\psi(t)^2)_{tt}}
-\uld{r}_{tt}=0
\end{equation}
\marginrr{3.p6-2}
in an $L^2(0,T;L^2(\Omega))$ sense onto the space-time basis functions $(t,x)\mapsto e^{\imath m\omega t}\varphi^{j,k}(x)$.
With the abbreviations
\begin{equation}\label{bmjkajk}
\begin{aligned}
&b_m^{j,k}=\langle \uld{\hat{u}}_m,\varphi^{j,k}\rangle, \quad 
a^{j,k}=\langle \phi^2\uld{\nlc},\varphi^{j,k}\rangle, \\
&o_m=\imath m\omega, \quad \mathfrak{B}_m=\mathcal{B}_m(\revision{\vec{\psi},\vec{\psi}}),
\end{aligned}
\end{equation}
\marginrr{3.p7-1}
and setting 
\begin{equation}\label{drjkm}
\uld{r}_m^{j,k}=\langle \uld{\hat{r}}_m,\varphi^{j,k}\rangle=\revision{\tfrac{2}{T}}\int_0^T\langle\uld{r}(t),\varphi^{j,k}\rangle \,e^{-\imath m \omega t}\, dt
\end{equation}
this yields
\begin{equation}\label{projmod}
\begin{aligned}
&\forall m\in\N\, \quad j\in\N, \quad k\in K^j: \\ 
&\bigl(\vartheta(o_m)+\Theta(o_m)\lambda_j\bigr)  b_m^{j,k} 
+\mathfrak{B}_m \, o_m^2\, a^{j,k} - \, o_m^2\, \uld{r}_m^{j,k}=0\revision{,}
\end{aligned}
\end{equation}
\marginR{4}
where in case of \eqref{JMGT-lin}
\begin{equation}\label{thetaJMGT}
\vartheta(o)=\tau o^3+o^2+\omzer o, \qquad
\Theta(o)=c^2(\tau o+1)+\delta o,
\end{equation}
but also other evolutionary models could be cast into the form \eqref{projmod} with different choices of the functions $\vartheta$, $\Theta$.

Equation \eqref{projmod} can be resolved for $b_m^{j,k}$
\begin{equation}\label{bmjk}
\forall m\in\N\, \quad j\in\N, \quad k\in K^j: 
  b_m^{j,k} = -
\frac{ o_m^2}{\vartheta(o_m)+\Theta(o_m)\lambda_j}(\mathfrak{B}_m \,a^{j,k} - \uld{r}_m^{j,k}).
\end{equation}
The linearized observation equation \revision{\eqref{IPfreq_diff_obs}}, 
\marginrr{3.p7-2}
after division by $\mathfrak{B}_m$ thus reads as
\begin{equation}\label{redobs}
\begin{aligned}
&\forall m\in\N\ : \quad\\
&-\sum_{j\in\N} \sum_{k\in K^j} \frac{o_m^2}{\vartheta(o_m)+\Theta(o_m)\lambda_j} \Bigl( a^{j,k} - \tfrac{1}{\mathfrak{B}_m}\uld{r}_m^{j,k}\Bigr)
\text{tr}_\Sigma\varphi^{j,k}(x) = \tfrac{1}{\mathfrak{B}_m} \uld{p}^{obs}_m(x).
\end{aligned}
\end{equation}

After this preparation, the linearized uniqueness and stability proof consists of two steps:
\begin{itemize}
\item decompose the observational data into its components in the $\Sigma$ traces of the eigenspaces of $\mathcal{A}$; this is achieved by taking the residues of an analytic extension of the observational data at poles defined by differential operators $\mathcal{L}_m$;
\item use unique continuation to recover each eigenspace component of $\uvec$ and $\nlc$.  
\end{itemize}
\subsubsection*{Step 1: Map temporal frequencies to $\Sigma$ traces of eigenspaces of $\mathcal{A}$.}
The key tool for disentangling the sum over $j$ in \eqref{redobs} will be the following lemma, which is a quantified and generalized version of \cite[Lemma 3.1]{nonlinearity_imaging_2d}.
\begin{lemma}\label{lem:m2j}
Given sequences $(\lambda_j)_{j\in\N}$ with $\lambda_j\nearrow\infty$ as $j\to\infty$, $(\hat{\rho}^j_m)_{j,m\in\N}$ and $(\hat{d}_m)_{m\in\N}$, as well as a complex function $\Psi$, assume that $\widetilde{d}$, $\widetilde{\rho^j}$ interpolate $\hat{d}$, $\hat{\rho}$, that is, 
\begin{equation}\label{interpol}
\widetilde{d}(\imath m\omega)=\hat{d}_m, \quad \widetilde{\rho^j}(\imath m\omega)=\hat{\rho}^j_m, \quad m\in\N
\end{equation} 
such that the mappings
\begin{equation}\label{analytic}
z\mapsto z^2\Psi(\tfrac{1}{z}), \quad z\mapsto\widetilde{\rho^j}(\tfrac{1}{z}), \quad z\mapsto(1-\tfrac{1}{\Psi(\tfrac{1}{z})}\lambda_j) \widetilde{d}(\tfrac{1}{z}), \quad z\mapsto\prod_{k\in\N,k\not=j} (1-\frac{1}{\Psi(\tfrac{1}{z})}\lambda_k) \  j\in\N
\end{equation} 
define analytic functions on an open set $\mathbb{O}\subseteq\mathbb{C}$.
Moreover, we assume that the sequence 
$(\pole_\ell)_{\ell\in\mathbb{N}}$ satisfies 
\[\Psi(\pole_\ell)=\lambda_\ell\] 
and that 
\begin{equation}\label{D}
\mathbb{D}:=\{\tfrac{1}{\imath m\omega}\, : \, m\in\N\}\subseteq \mathbb{O}
\text{ and }\{z_\ell:=\tfrac{1}{\pole_\ell}\, : \, \ell\in\N\}\subseteq \overline{\mathbb{O}}.
\end{equation}
Then 
\[
\begin{aligned}
&\left(\forall m\in\N\, : \ 
\sum_{j\in\N}\frac{(\imath m\omega)^2}{\Psi(\imath m\omega)-\lambda_j}\,\bigl(c_j-\hat{\rho}^j_m\bigr) = \hat{d}_m\right)
\\ 
&\Rightarrow \
\left(\forall \ell\in\N\, : \ c_\ell = \frac{\Psi'({\pole_\ell})}{\pole_\ell^2} \, \text{res}(\widetilde{d};{\pole_\ell})+\widetilde{\rho^\ell}({\pole_\ell})\right),
\end{aligned}
\]
where $\text{res}(f;p_\ell)=\lim_{o\to p_\ell} (o-p_\ell)\, f(o)$ is the residue of the function $f$ at $p_\ell$ if $p_\ell$ is a pole of $f$ and zero else.
\end{lemma}
\begin{proof}
We set 
\[
\Phi(z):=\frac{1}{\Psi(\tfrac{1}{z})}. \quad 
\]
Due to the analyticity assumption \eqref{analytic}, the interpolation condition \eqref{interpol}, which can be rewritten as collection of identities of analytic functions on a countably infinite set accumulating at zero
\[
\begin{aligned}
&\prod_{k\in\N,k\not=j} (1-\Phi(s^{(m)})\lambda_k)\, (1-\Phi(s^{(m)})\lambda_j)\widetilde{d}(\tfrac{1}{s^{(m)}})=
\prod_{k\in\N} (1-\Phi(\revision{s^{(m)})}
\lambda_k)\hat{d}_m, \quad s^{(m)}=\tfrac{1}{\imath m\omega}\in\mathbb{D}\\ 
&\widetilde{\rho^j}(\tfrac{1}{s^{(m)}})=\hat{\rho}^j_m, \quad s^{(m)}\in\mathbb{D}
\end{aligned}
\]
\marginrr{3.p8}
uniquely determines $\widetilde{d}$, $\widetilde{\rho^j}$, $j\in\N$.
Using this interpolating property of $\widetilde{d}$, $\widetilde{\rho^j}$, we can write the premiss as
\[
\forall z\in\{\tfrac{1}{\imath m\omega}\, : \, m\in\N\}\, : \ 
\sum_{j\in\N}\frac{1}{1-\Phi(z)\lambda_j}\bigl(c_j-\widetilde{\rho^j}(\tfrac{1}{z})\bigr) = 
\frac{z^2}{\Phi(z)} \widetilde{d}(\tfrac{1}{z}).
\]
Multiplying both sides with $\prod_{k\in\N} (1-\Phi(z)\lambda_k)$ we obtain
\[
\forall z\in \mathbb{D}\, : \ 
\sum_{j\in\N}\prod_{k\in\N,k\not=j} (1-\Phi(z)\lambda_k)\, \bigl(c_j-\widetilde{\rho^j}(\tfrac{1}{z})\bigr) = 
\frac{z^2}{\Phi(z)} \, \prod_{k\in\N} (1-\Phi(z)\lambda_k)\, \widetilde{d}(\tfrac{1}{z}),
\]
which can be rewritten as 
\[\begin{aligned}
&\forall\ell\in\N\, \forall z\in \mathbb{D}\, : \\ 
&\sum_{j\in\N}\prod_{k\in\N,k\not=j} (1-\Phi(z)\lambda_k)\, \bigl(c_j-\widetilde{\rho^j}(\tfrac{1}{z})\bigr) = 
\frac{z^2}{\Phi(z)} \, \prod_{k\in\N,k\not=\ell} (1-\Phi(z)\lambda_k)\, (1-\Phi(z)\lambda_\ell) \widetilde{d}(\tfrac{1}{z}).
\end{aligned}\]
Now, both sides are analytic functions of $z$ on $\mathbb{O}$ and due to \eqref{D}, $\mathbb{D}$ can thus be replaced by $\mathbb{O}$ here.\\
Taking limits as $z\to z_\ell:=\frac{1}{\pole_\ell}$ with $\Phi(z_\ell)\lambda_\ell=1$ (where we can rely on $\{z_\ell\, : \, \ell\in\N\}\subseteq \overline{\mathbb{O}}$), we see that all terms in the sum for $j\not=\ell$ contain a zero factor and obtain
\[
\forall\ell\in\N\, : \ 
\prod_{k\in\N,k\not=\ell} 
(1-\tfrac{\lambda_k}{\lambda_\ell})
\, \bigl(c_\ell-\widetilde{\rho^\ell}({\pole_\ell})\bigr) = 
\frac{z_\ell^2}{\Phi(z_\ell)} \, \prod_{k\in\N,k\not=\ell} 
(1-\tfrac{\lambda_k}{\lambda_\ell})
\, \lim_{z\to z_\ell}\,(1-\Phi(z)\lambda_\ell) \widetilde{d}(\tfrac{1}{z}).
\]
We divide by $\prod_{k\in\N,k\not=\ell} (1-\tfrac{\lambda_k}{\lambda_\ell})$ (which is nonzero since the $\lambda$'s are mutually different)
on both sides and use
\[
\lim_{z\to z_\ell}\,\frac{1-\Phi(z)\lambda_\ell}{\Phi(z_\ell)} \widetilde{d}(\tfrac{1}{z})
=\lim_{o\to {\pole_\ell}}\,\frac{\Psi({\pole_\ell})}{\Psi(o)}\,\frac{\Psi(o)-\lambda_\ell}{o-{\pole_\ell}} \, (o-{\pole_\ell}) \widetilde{d}(o)
=\Psi'({\pole_\ell}) \, \lim_{o\to {\pole_\ell}}(o-{\pole_\ell}) \widetilde{d}(o)
\]
to obtain the assertion.
\end{proof}

Now we apply Lemma~\ref{lem:m2j} to \eqref{redobs}.
With 
\begin{equation}\label{defPsiJMGT}
\begin{aligned}
&\vartheta(o)=\tau o^3+o^2+\omzer o, \quad \Theta(o)=c^2(\tau o+1)+\delta o\\
&\Psi(o)
=-\frac{\vartheta(o)}{\Theta(o)}
=-\frac{o^2+\frac{\omzer o}{\tau o+1}}{c^2+\frac{\delta o}{\tau o+1}}
=-\frac{o^2}{\tilde{c}^2}(1+\mathcal{O}(\tfrac{1}{o}))
\text{ with }\tilde{c}^2=c^2+\frac{\delta}{\tau}
\\
&c_j= \sum_{k\in K^j}  a^{j,k} \text{tr}_\Sigma\varphi^{j,k}(x), 
\\
&\hat{d}_m(x)=\tfrac{1}{\mathfrak{B}_m}\,\Theta(o_m)\,\uld{p}^{obs}_m(x)
, \quad\revision{\hat{\rho}}^j_m(x)=\tfrac{1}{\mathfrak{B}_m}\,\sum_{k\in K^j}\uld{r}_m^{j,k}\text{tr}_\Sigma\varphi^{j,k}(x) , \\ 
&\revision{\widetilde{d}}(o,x)= \tfrac{1}{\widetilde{\mathfrak{B}}(o)}\,\Theta(o)\,\widetilde{\uld{p}}^{obs}(o,x)
, \quad\revision{\widetilde{\rho^j}}(o,x)=\tfrac{1}{\widetilde{\mathfrak{B}}(o)}\,\sum_{k\in K^j}\widetilde{\uld{r}}^{j,k}(o)\text{tr}_\Sigma\varphi^{j,k}(x),\\
&\text{where for all }m\in\N\\
&\widetilde{\mathfrak{B}}(\imath m\omega)=\mathfrak{B}_m, \quad
\widetilde{\uld{p}}^{obs}(\imath m\omega,x)=\uld{p}^{obs}_m(x),\quad
\widetilde{\uld{r}}^{j,k}(\imath m\omega)=\uld{r}_m^{j,k},
\end{aligned}
\end{equation}
\marginrr{3.p9-1}
the reduced observation equation \eqref{redobs} can be written as the premiss of Lemma~\ref{lem:m2j}, which thus implies
\begin{equation}\label{resobs}
\sum_{k\in K^\ell} a^{\ell,k}  \text{tr}_\Sigma\varphi^{\ell,k}(x)
= \frac{1}{\widetilde{\mathfrak{B}}({\pole_\ell})}\left(
\frac{\Theta({\pole_\ell})\Psi'({\pole_\ell})}{\pole_\ell^2}
\, \text{res}(\widetilde{\uld{p}}^{obs}\revision{(\cdot,x)};{\pole_\ell})
+\sum_{k\in K^\ell} \widetilde{\uld{r}}^{\ell,k}({\pole_\ell})\text{tr}_\Sigma\varphi^{\ell,k}(x)\right).
\end{equation}
\marginrr{3.p9-2}
Here we have 
\[
\Psi(\pole_\ell)=\lambda_\ell \ \Leftrightarrow \ 
\pole_\ell^2 = -\frac{\omzer}{\tau+\frac{1}{\pole_\ell}}-\left(c^2+\frac{\delta}{\tau+\frac{1}{\pole_\ell}}\right)\lambda_\ell,
\]
thus, $\frac{1}{\pole_\ell}=O(\lambda_\ell^{-1/2})$ as $\lambda_\ell\to\infty$ and therefore $\pole_\ell^2 = -\tilde{c}^2\lambda_\ell-\frac{\omzer}{\tau}(1+O(\lambda_\ell^{-1/2}))+O(\delta\lambda_\ell^{1/2})$
\begin{equation}\label{zell_JMGT}
\begin{aligned}
&{\pole_\ell}=\imath \left((c^2+\tfrac{\delta}{\tau})\lambda_\ell+\tfrac{\omzer}{\tau}\right)^{1/2}+\mathcal{O}(\omzer^{1/2}\lambda_\ell^{-1/4})+\mathcal{O}(\delta^{1/2}\lambda_\ell^{1/4})\\
&\Theta(o)\Psi'(o)=-\vartheta'(o)+\vartheta(o)\frac{\Theta'(o)}{\Theta(o)}
=-2\tau o^2-o-\frac{c^2\revision{(\tau o^2+o+\omzer)}}{(\tau c^2+\delta)o+c^2},
\end{aligned}
\end{equation}
\marginrr{3.p10-1}
hence the amplification factor
\begin{equation}\label{amplfac_JMGT}
\frac{\Theta({\pole_\ell})\Psi'({\pole_\ell})}{\pole_\ell^2}
=-2\tau-\tfrac{1}{\pole_\ell}-\frac{c^2\revision{(\tau \pole_\ell^2+\pole_\ell+\omzer)}}{(\tau c^2+\delta)\pole_\ell^3+c^2\pole_\ell}
=-2\tau +\mathcal{O}(\lambda_\ell^{-1/2}),
\end{equation}
where the Landau symbol is to be understood in the sense of $\lambda_\ell^{-1/2}$ 
tending to zero.

The interpolants \revision{$\widetilde{\uld{r}}^{j,k}$, $\widetilde{\mathfrak{B}}$ of $(\uld{r}_m^{j,k})_{m\in\mathbb{N}}$, $(\mathfrak{B}_m)_{m\in\mathbb{N}}$}, due to the identities \eqref{Bm}, \eqref{drjkm} can be explicitely written as 
\marginrr{3.p10-2}
\begin{equation}\label{interpol_r}
\widetilde{\uld{r}}^{j,k}(
o)=\revision{\tfrac{2}{T}}\int_0^T\langle\uld{r}(
t),\varphi^{j,k}\rangle \,e^{-o t}\,dt,
\quad \widetilde{\mathfrak{B}}(o)=\revision{\tfrac{2}{T}}\int_0^T \psi^2(t) \,e^{-o t}\,dt,
\end{equation}
which defines analytic functions, provided $\uld{r}
$ and $\psi^2$ are integrable.
From this using the Cauchy-Schwarz \revision{inequality}, 
\marginrr{3.p10-3}
one can directly derive the estimate
\begin{equation}\label{est_interpol_r}
\begin{aligned}
\|\sum_{k\in K^\ell} \widetilde{\uld{r}}^{\ell,k}({\pole_\ell})\,\varphi^{\ell,k}\|_{H^{\ddot{s}}(\Omega)}
&=\|\revision{\tfrac{2}{T}}\int_0^T\text{Proj}_{\revision{\mathbb{E}^\ell}}\uld{r}(t) \,e^{-{\pole_\ell} t}\,dt\|_{H^{\ddot{s}}(\Omega)}\\
&\leq \frac{e^{-\Re({\pole_\ell}) T}}{\sqrt{-2\Re({\pole_\ell})}T}
\|\text{Proj}_{\revision{\mathbb{E}^\ell}}\uld{r}\|_{L^2(0,T;H^{\ddot{s}}(\Omega))}
\end{aligned}
\end{equation}
for any $\ddot{s}\in\mathbb{R}$, where we have used the fact that $\Re({\pole_\ell})\leq0$ holds according to \cite[Lemma 5.1]{nonlinearity_imaging_fracWest}.
An enhanced estimate can be obtained from integration by parts 
\marginrr{3.p10-4}
\begin{equation}\label{est_interpol_r1}
\begin{aligned}
&\|\sum_{k\in K^\ell} \widetilde{\uld{r}}^{\ell,k}({\pole_\ell})\,\varphi^{\ell,k}\|_{H^{\ddot{s}}(\Omega)}\\
&=\|-\revision{\tfrac{2}{T}}\int_0^T\text{Proj}_{\revision{\mathbb{E}^\ell}}\uld{r}_t(t) \,\tfrac{e^{-{\pole_\ell} t}}{-{\pole_\ell}}\,dt
+\revision{\tfrac{2}{T}}\left[\text{Proj}_{\revision{\mathbb{E}^\ell}}\uld{r}(t) \,\tfrac{e^{-{\pole_\ell} t}}{-{\pole_\ell}}\right]_0^T\|_{H^{\ddot{s}}(\Omega)}\\
&\leq 
\revision{2}\bigl(\revision{(-q\Re({\pole_\ell}))^{-1/q}} +2C_{W^{1,q^*},L^\infty}^{(0,T)}\bigr)C_{H^{\ddot{\sigma}},W^{1,q^*}}^{(0,T)}\,
\frac{e^{-\Re({\pole_\ell}) T}}{|{\pole_\ell}|T}
\|\text{Proj}_{\revision{\mathbb{E}^\ell}}\uld{r}\|_{H^{\ddot{\sigma}}(0,T;H^{\ddot{s}}(\Omega))},
\end{aligned}
\end{equation}
\marginrr{2.2.2}
where $C_{H^{\ddot{\sigma}},W^{1,q^*}}^{(0,T)}$ is the norm of the embedding $H^{\ddot{\sigma}}(0,T)\to W^{1,q^*}(0,T)$ for 
\[
1<q^*:=\frac{q}{q-1}, \quad \ddot{\sigma}\geq\max\{1,\tfrac12+\tfrac{1}{q}\}.
\]
These estimates clearly show the relevance of the position of the poles ${\pole_\ell}$. In particular, the further \revision{to the} 
\marginrr{3.p11-1}
left of the imaginary axis (where values are given at $o_m=\imath m \omega$) these poles lie, the larger the amplification factor containing $e^{-\Re({\pole_\ell}) T}$ will be.
  
For $\text{res}(\widetilde{\uld{p}}^{obs}(\cdot,x);{\pole_\ell})$ 
\revision{neither such an} 
\marginrr{3.p11-2}
explicit expression of the interpolant (which is not analytic but has poles and therefore cannot be expressed by taking $L^2(0,T)$ inner products with the exponential functions $e^{-o t}$) nor an explicit bound like \eqref{est_interpol_r} are available.

\subsubsection*{Step 2: Map traces of eigenspaces of $\mathcal{A}$ to coefficients of $\nlc$ (and $u$).}
\begin{lemma}\label{lem:trace2coeff}
Let 
\revision{$s\in(\frac12,\infty)$,} 
\marginrr{2.2.1}
$\Sigma\subseteq\overline{\Omega}\subseteq\mathbb{R}^d$, $d\in\{2,3\}$ be a $C^{\max\{s,1\}}$ manifold with $\text{meas}^{d-1}(\Sigma)>0$, 
and denote by $\|\cdot\|_{h^{s,j}}$ the weighted Euclidean norm defined by  
\[
\|(a^k)_{k\in K^j}\|_{h^{s,j}}=\left((\lambda^j)^s\sum_{k\in K^j}|a^k|^2\right)^{1/2}
\]
\\
Then, \revision{for each $j\in\N$,} the operator 
\[
\begin{array}{rlcl}\text{Tr}_\Sigma^{s,j}:
&(\mathbb{C}^{K^j},\|\cdot\|_{h^{s,j}})&\to& 
(\text{tr}_\Sigma(\mathbb{E}^j), \revision{\|\cdot\|_{H^{s-1/2}(\Sigma)}})\\
&(a^k)_{k\in K^j}&\mapsto&\sum_{k\in K^j} a^k \text{tr}_\Sigma \varphi^{j,k}
\end{array} 
\]
\marginrr{3.p11-3}
is an isomorphism. 
\end{lemma}
\begin{proof}
Using the norm $\|v\|_{H^s(\Omega)}:=\left(\sum_{j\in\N}(\lambda^j)^s\sum_{k\in K^j}|\langle v,\varphi^{j,k}\rangle|^2\right)^{1/2}$ and the isometric isomorphism  
\[
\begin{array}{rlcl}
&(\mathbb{C}^{K^j},\|\cdot\|_{h^{s,j}})&\to& (\mathbb{E}^j, \|\cdot\|_{H^s(\Omega)})\\
&(a^k)_{k\in K^j}&\mapsto&\sum_{k\in K^j} a^k \varphi^{j,k}
\end{array} 
\]
we reduce the proof to showing that  
\[
\begin{array}{rlcl}\widetilde{\text{Tr}_\Sigma}^{s,j}:
&(\mathbb{E}^j, \|\cdot\|_{H^s(\Omega)})&\to& (\text{tr}_\Sigma(\mathbb{E}^j), 
\revision{\|\cdot\|_{H^{s-1/2}(\Sigma)}})\\
&v&\mapsto&\text{tr}_\Sigma v
\end{array} 
\]
is an isomorphism. 
With the Bounded Inverse Theorem on the Banach spaces $(\mathbb{E}^j, \|\cdot\|_{H^s(\Omega)})$, $(\text{tr}_\Sigma(\mathbb{E}^j), H^{s-1/2}(\Sigma))$, this is an immediate consequence of boundedness of $\widetilde{\text{Tr}_\Sigma}^{s,j}$ due to the Trace Theorem, as well as its injectivity due to unique continuation \cite{Aronszajn:1957,JiangLiPauronYamamoto2023,Tataru:1995,Tolsa:2023}.
\end{proof}

Applying Lemma~\ref{lem:trace2coeff} to \eqref{resobs} and using the fact that 
$(\text{Tr}_\Sigma^{s,j})^{-1}\, \text{tr}_\Sigma\, \text{Proj}_{\mathbb{E}^j} = \text{Proj}_{\mathbb{E}^j}$ as well as \eqref{interpol_r}, we obtain
\begin{equation}\label{resobs_1}
\begin{aligned}
(a^{\ell,k})_{k\in K^\ell}
=& (\text{Tr}_\Sigma^{s,j})^{-1}\,\left[\frac{1}{\widetilde{\mathfrak{B}}({\pole_\ell})}
\frac{\Theta({\pole_\ell})\Psi'({\pole_\ell})}{\pole_\ell^2}
\, \text{res}(\widetilde{\uld{p}}^{obs};{\pole_\ell})\right]
\\&\qquad
+\frac{1}{\widetilde{\mathfrak{B}}({\pole_\ell})} 
\text{Proj}_{\mathbb{E}^j} \revision{\tfrac{2}{T}}\int_0^T \uld{r}(t) \,e^{-\revision{\pole_\ell} t}\,dt,
\end{aligned}
\end{equation}
\marginrr{3.p11-4}

With $s>\frac12$ as in Lemma~\ref{lem:trace2coeff}, $\check{\sigma}$, $\check{s}$, $\ddot{\sigma}$, $\ddot{s}$  $\in\mathbb{R}$, and the abbreviations
\begin{equation}\label{constants}
\begin{aligned}
&\bar{C}_{\vartheta\Theta}
:=\sup_{\ell\in\mathbb{N}} \left|\frac{\Theta({\pole_\ell})\Psi'({\pole_\ell})}{\pole_\ell^2}\right|
=\sup_{\ell\in\mathbb{N}} \left|\frac{\vartheta'({\pole_\ell})+\lambda_\ell\Theta'({\pole_\ell})}{\pole_\ell^2}\right|
\\
&\check{C}^{(\check{s}-s,\check{\sigma})}_{\vartheta\Theta}
:=\left(\sum_{m\in\N}\sup_{j\in\N} 
\left\{
(\lambda^j)^{\check{s}-s}
\left|\frac{o_m^{2+\check{\sigma}}\mathfrak{B}_m}{\vartheta(o_m)+\Theta(o_m)\lambda_j}\right|^2
\right\}
\right)^{1/2}\\
& \hat{C}_{\vartheta\Theta}^{(\check{s}-\ddot{s},\check{\sigma}-\ddot{\sigma})}
:=
\sup_{m\in\N} \sup_{j\in\N} (\lambda^j)^{(\check{s}-\ddot{s})/2}\left|\frac{o_m^{2+\check{\sigma}-\ddot{\sigma}}}{\vartheta(o_m)+\Theta(o_m)\lambda_j}\right|
\end{aligned}
\end{equation}
we obtain from \eqref{resobs_1}
\begin{equation}\label{stabest_nlc}
\begin{aligned}
&\|\phi^2\uld{\nlc}\|_{H^s(\Omega)}
=\left(\sum_{j\in\N}(\lambda^j)^s\sum_{k\in K^j} |a^{j,k}|^2\right)^{1/2}\\
&\leq \bar{C}_{\vartheta\Theta}|\uld{\pvec}^{obs}|_{\YYY^{obs}_{pol}} + |\uld{\rvec}|_{\YYY^{mod}_{pol}}
\end{aligned}
\end{equation}
and from \eqref{bmjk}
\begin{equation}\label{stabest_u}
\begin{aligned}
&\|\uld{u}\|_{H^{\check{\sigma}}(0,T;H^{\check{s}}(\Omega))}
=\|\uld{\uvec}\|_{h^{\check{\sigma}}(H^{\check{s}}(\Omega))}
=\left(\sum_{m\in\N}|o_m|^{2\check{\sigma}}  \sum_{j\in\N}\sum_{k\in K^j} (\lambda^j)^{\check{s}} |b_m^{j,k}|^2\right)^{1/2}\\
&= 
\left(\sum_{m\in\N}|o_m|^{2\check{\sigma}} \sum_{j\in\N} (\lambda^j)^{\check{s}}\sum_{k\in K^j} 
\left|\frac{o_m^2\Bigl(\mathfrak{B}_m a^{j,k} - \uld{r}_m^{j,k}\Bigr)}{\vartheta(o_m)+\Theta(o_m)\lambda_j}\right|^2
\right)^{1/2}\\
&\leq 
\check{C}^{(\check{s}-s,\check{\sigma})}_{\vartheta\Theta} \|\phi^2\uld{\nlc}\|_{H^s(\Omega)}+|\uld{\rvec}|_{\YYY^{mod}_{Sob}}\\
&\leq 
\check{C}^{(\check{s}-s,\check{\sigma})}_{\vartheta\Theta} \|\phi^2\uld{\nlc}\|_{H^s(\Omega)}
+\hat{C}^{(\check{s}-\ddot{s},\check{\sigma}-\ddot{\sigma})}_{\vartheta\Theta}|\uld{\rvec}|_{h^{\ddot{\sigma}}(h^{\ddot{s}})}\revision{.}
\end{aligned}
\end{equation}
\marginR{5.}
In image space, on the one hand we use norms defined via the poles and interpolation
\marginrr{3.p12}
\begin{equation}\label{Ynorms_pol}
\begin{aligned} 
|\uld{\pvec}^{obs}|_{\YYY^{obs}_{pol}}&
:=\left(\sum_{j\in\N} 
\frac{\|(\text{Tr}_\Sigma^{s,j})^{-1}\|_{\revision{L(H^{s-1/2}(\Sigma)\to h^{s,j})}}^2}{\widetilde{\mathfrak{B}}({\pole_j})^2}
\left(\|\text{res}(\widetilde{\uld{p}}^{obs};{\pole_j})\|_{H^{s-1/2}(\Sigma)}\right)^2\right)^{1/2}
\\
|\uld{\rvec}|_{\YYY^{mod}_{pol}}&
:=\left(\sum_{j\in\N} 
\frac{1}{\widetilde{\mathfrak{B}}({\pole_j})^2}
\|\revision{\tfrac{2}{T}}\int_0^T\text{Proj}_{\revision{\mathbb{E}^\ell}}\uld{r}(t) \,e^{-{\pole_j} t}\,dt\|_{H^s(\Omega)}^2
\right)^{1/2},
\end{aligned}
\end{equation}
where the latter can be estimated by \eqref{est_interpol_r} \revision{ or \eqref{est_interpol_r1}};
on the other hand the Sobolev-Bochner type norm (cf. the last line of \eqref{constants})
\begin{equation}\label{YSob}
\begin{aligned} 
|\uld{\rvec}|_{\YYY^{mod}_{Sob}}&
:=\left(\sum_{m\in\N}|o_m|^{2\check{\sigma}}\sum_{j\in\N}(\lambda^j)^{\check{s}}\sum_{k\in K^j} 
\left|\frac{o_m^2\uld{r}_m^{j,k}}{\vartheta(o_m)+\Theta(o_m)\lambda_j}\right|^2
\right)^{1/2}\\
&\leq \hat{C}_{\vartheta,\Theta}^{(\check{s}-\ddot{s},\check{\sigma}-\ddot{\sigma})}\left(\sum_{m\in\N}|o_m|^{2\ddot{\sigma}}
\sum_{j\in\N}(\lambda^j)^{\ddot{s}}\sum_{k\in K^j} |\uld{r}_m^{j,k}|^2
\right)^{1/2}\\
&= \hat{C}_{\vartheta,\Theta}^{(\check{s}-\ddot{s},\check{\sigma}-\ddot{\sigma})}
\|\uld{r}\|_{H^{\ddot{\sigma}}(0,T;H^{\ddot{s}}(\Omega)\revision{)}}.
\end{aligned}
\end{equation}
\marginrr{3.p13-1}
Note that due to \eqref{est_interpol_r}, \eqref{est_interpol_r1}, under the condition 
\begin{equation}\label{amplfac_lambda}
\frac{1}{|\widetilde{\mathfrak{B}}({\pole_j})|} \,
\frac{e^{-\Re({\pole_j}) T}}{\sqrt{-\revision{2}\Re({\pole_j})}T}\, \lambda_j^{(s-\ddot{s})/2}\leq \bar{M}  
\quad\text{ and }\ddot{\sigma}\geq0
\end{equation}
\marginrr{3.p13-2}
or 
\begin{equation}\label{amplfac_lambda1}
\frac{1}{|\widetilde{\mathfrak{B}}({\pole_j})|} \,
\frac{e^{-\Re({\pole_j}) T}}{|\pole_j|T}\, \lambda_j^{(s-\ddot{s})/2}\leq \bar{M}  
\quad\text{ and }\ddot{\sigma}\geq\max\{1,\tfrac12+\tfrac{1}{q}\},\quad q<\infty,
\end{equation}
we can also bound the pole related norm of the PDE residual 
by its Sobolev norm
\begin{equation}\label{Ymodpol_YmodSob}
|\uld{\rvec}|_{\YYY^{mod}_{pol}}
\leq \revision{\bar{C}}
\|\uld{r}\|_{H^{\ddot{\sigma}}(0,T;H^{\ddot{s}}(\Omega)\revision{)}}.
\end{equation}
\marginR{6.}
\revision{for some constant $\bar{C}>0$, provided the real parts of the poles are bouded away from zero.}
\marginrr{3.p13-3}
\marginrr{3.p13-4}

Combining the estimates \eqref{stabest_nlc}, \eqref{stabest_u}, we obtain 
\[
\begin{aligned}
&\|\phi^2\uld{\nlc}\|_{H^s(\Omega)}
+\tfrac{1}{2\check{C}_{\vartheta\Theta}^{(\check{s}-s,\check{\sigma})}}\|\uld{\uvec}\|_{h^{\check{\sigma}}(H^{\check{s}}(\Omega))}\\
&\leq \bar{C}_{\vartheta\Theta}|\uld{\pvec}^{obs}|_{\YYY^{obs}_{pol}} + |\uld{\rvec}|_{\YYY^{mod}_{pol}}
+\tfrac12 \|\phi^2\uld{\nlc}\|_{H^s(\Omega)}+\tfrac{1}{2\check{C}^{(\check{s}-s,\check{\sigma})}_{\vartheta\Theta}}|\uld{\rvec}|_{\YYY^{mod}_{Sob}},
\end{aligned}
\]
hence the following stability estimate.
\begin{theorem}\label{thm:uniqueness-stability_lin}
With $\XXX=\revision{H^s_\phi(\Omega)}\times h^{\check{\sigma}}(H^{\check{s}}(\Omega))$,  
\revision{where $\|h\|_{H^s_\phi(\Omega)}:=\|\phi^2 h\|_{H^s(\Omega)}$}
and 
\[
\begin{aligned}
\|(\uld{\nlc},\uld{\uvec})\|_{\XXX}&=
\revision{\|\uld{\nlc}\|_{H^s_\phi(\Omega)}}
+\tfrac{1}{\check{C}_{\vartheta\Theta}^{(\check{s}-s,\check{\sigma})}}\|\uld{\uvec}\|_{h^{\check{\sigma}}(H^{\check{s}}(\Omega))}\\ 
\|(\uld{\pvec}^{obs},\uld{\rvec})\|_{\YYY}&=
2\bar{C}_{\vartheta\Theta}|\uld{\pvec}^{obs}|_{\YYY^{obs}_{pol}}+ 
\underbrace{2|\uld{\rvec}|_{\YYY^{mod}_{pol}}+\tfrac{1}{\check{C}_{\vartheta\Theta}^{(\check{s}-s,\check{\sigma})}}|\uld{\rvec}|_{\YYY^{mod}_{Sob}}}_{\revision{=:|\uld{\rvec}|_{\YYY^{mod}}}}
\end{aligned}
\]
\marginrr{2.3.1(a)}
\marginrr{2.3.2(b)}
the stability estimate 
\begin{equation}\label{stabest_nlcu}
\|(\uld{\nlc},\uld{\uvec})\|_{\XXX}
\leq \|\vec{F}'(0,\uvec^0)[(\uld{\nlc},\uld{\uvec})]\|_{\YYY}
\end{equation}
holds.
\end{theorem}

\subsection{Injectivity of (nonlinear) all-at-once forward operator}\label{sec:uniqueness-nonlin}
We now obtain local uniqueness and stability for the fully nonlinear problem by a simple perturbation argument, cf. \eqref{stabest0} -- \eqref{stabest_nlcu1} below.
To quantify this in Sobolev spaces, some more technical estimates will be needed, cf. \eqref{errTay_Sob} -- \eqref{estTay1} below.

To this end, we again use the abbreviation $x=(\revision{\nlc},\uvec)$.
\marginrr{3.p14-1}
Since by Theorem~\ref{thm:uniqueness-stability_lin}
\[
\|\vec{F}'(x^0)(\tilde{x}-x)\|_\YYY 
\geq \frac{1}{\|\vec{F}'(x^0)^{-1}\|_{\YYY\to\XXX}} \|\tilde{x}-x\|_\XXX \quad\text{ with }\|\vec{F}'(x^0)^{-1}\|_{\YYY\to\XXX}\leq1
\]
holds, where $\XXX$, $\YYY$ are defined as in Theorem~\ref{thm:uniqueness-stability_lin}, we obtain
the stability estimate 
\begin{equation}\label{stabest0}
\|\vec{F}(\tilde{x})-\vec{F}(x)\|_\YYY \geq 
\left(\frac{1}{\|\vec{F}'(x^0)^{-1}\|_{\YYY\to\XXX}}-c\right) \|\tilde{x}-x\|_\XXX
\end{equation}
provided the Taylor remainder estimate 
\begin{equation}\label{Taylorremainder}
\text{err}_{Tay}:=\|\vec{F}(\tilde{x})-\vec{F}(x)-\vec{F}'(x^0)(\tilde{x}-x)\|_\YYY \leq c \|\tilde{x}-x\|_\XXX
\end{equation}
holds with $c>1$.
Note that this argument is similar but not identical to the Inverse Function Theorem, since here Fr\'{e}chet differentiability only holds at the point $(0,\uvec^0)$.
Since the observation equation is linear, the Taylor remainder only consists of the model part
\[
\begin{aligned}
&
\tilde{\nlc} \mathcal{B}_m(\tilde{\uvec},\tilde{\uvec})
-\nlc \mathcal{B}_m(\uvec,\uvec)
-(\tilde{\nlc}-\nlc) \mathcal{B}_m(\uvec^0,\uvec^0)\\
&= \uld{\nlc} (\mathcal{B}_m(\tilde{\uvec},\tilde{\uvec}-\uvec^0)+\mathcal{B}_m(\tilde{\uvec}-\uvec^0,\uvec^0))
+\nlc (\mathcal{B}_m(\tilde{\uvec},\uld{\uvec})+\mathcal{B}_m(\uld{\uvec},\uvec))\\
&=\mathcal{Q}_m^\nlc(\uld{\nlc};\tilde{\uvec}-\uvec^0)+\mathcal{Q}_m^u(\uld{\uvec};\tilde{\uvec}-\uvec^0,\eta)\\
&\text{with }\mathcal{Q}_m^\nlc(\uld{\nlc};\tilde{\uvec}-\uvec^0)
=\uld{\nlc} (\mathcal{B}_m(\tilde{\uvec}-\uvec^0,\tilde{\uvec}-\uvec^0)+\mathcal{B}_m(\uvec^0,\tilde{\uvec}-\uvec^0)+\mathcal{B}_m(\tilde{\uvec}-\uvec^0,\uvec^0))\\
&\hspace*{1cm}\mathcal{Q}_m^u(\uld{\uvec};\tilde{\uvec}-\uvec^0,\eta)
=\nlc (\mathcal{B}_m(\uvec^0,\uld{\uvec})+\mathcal{B}_m(\uld{\uvec},\uvec^0)+\mathcal{B}_m(\tilde{\uvec}-\uvec^0,\uld{\uvec})\\
&\hspace*{10cm}
+\mathcal{B}_m(\uld{\uvec},\tilde{\uvec}-\uld{\uvec}-\uvec^0))
\end{aligned}
\]
for $\uld{\nlc}=\tilde{\nlc}-\nlc$, $\uld{\uvec}=\tilde{\uvec}-\uvec$.
Thus, condition \eqref{Taylorremainder} can be interpreted as closeness of $(\nlc,\uvec,\tilde{\uvec})$ to $(0,\uvec^0,\uvec^0)$ in the sense that the 
\revision{function} 
\begin{equation}\label{closeness}
\begin{aligned}
&d\bigl((\nlc,\tilde{\uvec}),(0,\uvec^0)\bigr)\\&:=
\sup_{h\in H^s(\Omega)\setminus\{0\}}
\frac{|\mathcal{Q}_m^\nlc(h;\tilde{\uvec}-\uvec^0)|_{\YYY^{mod}}}{\|h\|_{\revision{H^s_\phi}}} 
+\revision{\check{C}_{\vartheta\Theta}^{(\check{s}-s,\check{\sigma})}}
\sup_{h\in h^{\check{\sigma}};H^{\check{s}}(\Omega))\setminus\{0\}}
\frac{|\mathcal{Q}_m^u(h;\tilde{\uvec}-\uvec^0,\eta)|_{\YYY^{mod}}
}{\|h\|_{h^{\check{\sigma}}\revision{(}H^{\check{s}}(\Omega))}} 
\\
&<1.
\end{aligned}
\end{equation}
\marginrr{2.3.1.(a)}
\marginrr{3.p14-2}
\revision{By definition we thus have 
\[
\text{err}_{Tay}=|\tilde{\nlc} \mathcal{B}_m(\tilde{\uvec},\tilde{\uvec})
-\nlc \mathcal{B}_m(\uvec,\uvec)
-(\tilde{\nlc}-\nlc) \mathcal{B}_m(\uvec^0,\uvec^0)_{\YYY^{mod}}
\leq d\bigl((\nlc,\tilde{\uvec}),(0,\uvec^0)\bigr) \|\tilde{x}-x\|_\XXX.
\]
}
\marginrr{2.3.2.(a)}
\marginrr{2.3.2.(b)\!}
\revision{It is on purpose that $d$ notationally appears as a distance between the two points $(\nlc,\tilde{\uvec})$ and $(0,\uvec^0)$, although it formally does not satisfy the axioms of a metric.
In the same spirit, we think of \eqref{closeness} being a closeness assumption of $(\nlc,\tilde{\uvec})$ and $(0,\uvec^0)$. Indeed, in Corollary~\ref{cor:stabSobolev}, we will put the stability estimate into a framework of (Sobolev) norms. 
}
We thus obtain 
the following stability estimate.
\begin{theorem}\label{thm:uniqueness-stability}
For all $(\nlc,\tilde{\uvec})$ sufficiently close to $(0,\uvec^0)$ in the sense of \eqref{closeness}, the stability and uniqueness estimate
\begin{equation}\label{stabest_nlcu1}
\|(\tilde{\nlc}-\nlc,\tilde{\uvec}-\uvec)\|_\XXX
\leq C\|\vec{F}(\tilde{\nlc},\tilde{\uvec})-\vec{F}(\nlc,\uvec)\|_\YYY
\end{equation}
with $C=\frac{1}{1-d\bigl((\nlc,\tilde{\uvec}),(0,\uvec^0)\bigr)}$
holds.
\end{theorem}
\revision{
\begin{remark}\label{rem:stabest_nlcu_ext}
Similarly, the linearized stability estimate \eqref{stabest_nlcu} extends to  all $(\nlc,\tilde{\uvec})$ such that 
\begin{equation}\label{closeness_til}
\begin{aligned}
&\tilde{d}\bigl((\nlc,\tilde{\uvec}),(0,\uvec^0)\bigr)\\&:=
\sup_{h\in H^s(\Omega)\setminus\{0\}}
\frac{|\mathcal{Q}_m^\nlc(h;\tilde{\uvec}-\uvec^0)|_{\YYY^{mod}}}{\|h\|_{\revision{H^s_\phi}}} 
+\revision{\check{C}_{\vartheta\Theta}^{(\check{s}-s,\check{\sigma})}}
\sup_{h\in h^{\check{\sigma}};H^{\check{s}}(\Omega))\setminus\{0\}}
\frac{|\widetilde{\mathcal{Q}}_m^u(h;\tilde{\uvec},\eta)|_{\YYY^{mod}}
}{\|h\|_{h^{\check{\sigma}};H^{\check{s}}(\Omega))}} 
\\
&<1.
\end{aligned}
\end{equation}
with $\widetilde{\mathcal{Q}}_m^u(h;\tilde{\uvec},\eta)
=\nlc (\mathcal{B}_m(\uld{\uvec},\tilde{\uvec})+\mathcal{B}_m(\tilde{\uvec},\uld{\uvec}))$
holds, in the sense that 
\begin{equation}\label{stabest_nlcu_ext}
\|(\uld{\nlc},\uld{\uvec})\|_{\XXX}
\leq C \|\vec{F}'(\nlc,\tilde{\uvec})[(\uld{\nlc},\uld{\uvec})]\|_{\YYY}
\end{equation}
holds with $C=\frac{1}{1-\tilde{d}\bigl((\nlc,\tilde{\uvec}),(0,\uvec^0)\bigr)}$.
\end{remark}
}
\marginrr{2.3.3}
A sufficient condition for \eqref{closeness} can be obtained by further bounding the Taylor remainder.
We can do so in case both $\revision{\YYY^{mod}}$ 
\marginrr{3.p15-1}
norms can be chosen as Sobolev type, that is, under condition \eqref{amplfac_lambda} or \eqref{amplfac_lambda1}. Then we have
\begin{equation}\label{errTay_Sob}
\begin{aligned}
\text{err}_{Tay}&\leq \tilde{C}
\|\tilde{\nlc} \tilde{u}^2 -\nlc u^2 -(\tilde{\nlc}-\nlc) {u^0}^2\|_{H^{\ddot{\sigma}}(0,T;H^{\ddot{s}}(\Omega))}\\
&=\tilde{C}\|\uld{\nlc}\,(\tilde{u}+{u^0})(\tilde{u}-u^0)
+\nlc\,(\tilde{u}+u)\uld{u}\|_{H^{\ddot{\sigma}}(H^{\ddot{s}})}\\
\end{aligned}
\end{equation}
for 
\revision{$\tilde{C}=\frac{\hat{C}_{\vartheta,\Theta}^{(\check{s}-\ddot{s},\check{\sigma}-\ddot{\sigma})}}{\check{C}_{\vartheta\Theta}^{(\check{s}-s,\check{\sigma})}}+2\bar{C}$}
\marginrr{3.p15-2}
(cf. \eqref{YSob}, \eqref{Ymodpol_YmodSob}), $\uld{\nlc}=\tilde{\nlc}-\nlc$, $\uld{u}=\tilde{u}-u$.
\marginrr{3.p15-3}
In order to bound \eqref{errTay_Sob} in terms of the 
\revision{$H^s_\phi(\Omega)$} 
\marginrr{2.3.1.(b)\!}
norm of $\uld{\nlc}$ and the $ H^{\check{\sigma}}(0,T;H^{\check{s}}(\Omega))$ norm of $\uld{u}$, a necessary condition is obviously 
\begin{equation}\label{cond_ssigma} 
s\geq\ddot{s}, \quad \check{s}\geq\ddot{s}, \quad \check{\sigma}\geq\ddot{\sigma}.
\end{equation}
\revision{
Moreover, to obtain equivalence of $H^s_\phi(\Omega)$ to $H^s(\Omega)$, we assume $\phi$ to be sufficiently regular and bounded away from zero and infinity\footnote{Indeed, the above equivalence statement is then a consequence of \eqref{KatoPonce}
with $s$ in place of $\ddot{s}$, $r=p=2$, $q=\infty$, $s_1=s$, $s_2=0$}
\begin{equation}\label{phiLinfty}
\phi^2, \ \frac{1}{\phi^2}\in W^{s,\infty}(\Omega);
\end{equation}
}

\marginrr{2.3.1.(a)}
To further estimate \eqref{errTay_Sob}, which contains products of functions, we apply the Kato-Ponce inequality \cite{KenigPonceVega1993}
\begin{equation}\label{KatoPonce}
\revision{\|}\mathcal{L}^{\ddot{s}}[f\cdot g]-f\mathcal{L}^{\ddot{s}} g-g\mathcal{L}^{\ddot{s}} f\revision{\|}_{L^r(\Omega)}
\leq C^d_{KP}\revision{\|}\mathcal{L}^{s_1} f\revision{\|}_{L^p(\Omega)}\revision{\|}\mathcal{L}^{s_2} g\revision{\|}_{L^q(\Omega)}
\end{equation}
\marginrr{3.p15-4}
for $C^d_{KP}=C^d_{KP}(\revision{\ddot{s}},s_1,s_2,r,p,q)$, ${\ddot{s}},s_1,s_2\in[0,1)$, ${\ddot{s}}=s_1+s_2$, $p,q,r\in(1,\infty]$, $\tfrac{1}{r}=\tfrac{1}{p}+\tfrac{1}{q}$, $\mathcal{L}=(-\Delta)^{1/2}$.
\marginrr{3.p15-5}

\def\aaa{\mathfrak{a}}
\def\bbb{\mathfrak{b}}
\def\ptil{{\tilde{p}}}
\def\stil{{\tilde{s}}}
\def\pbar{{\bar{p}}}
\def\sbar{{\bar{s}}}

We first of all consider the case $\ddot{\sigma}=0$ 
and with $\ddot{s}\geq0$ and 
\[
\aaa=(\tilde{u}+u^0)(\tilde{u}-u^0), \quad \bbb=\nlc(\tilde{u}+u), 
\]
as well as H\"older's inequality
\begin{equation}\label{psharp}
\|fg\|_{L^2}\leq\|f\|_{L^p}\|g\|_{L^{p^\sharp}}\quad\text{ with }
p^\sharp:=\begin{cases}\frac{2p}{p-2}&\text{if }p>2\\ \infty &\text{if }p=2\end{cases}, \quad \text{thus }(p^\sharp)^\sharp=p, 
\end{equation}
\marginrr{3.p15-6}
we have
\marginrr{3.p15-7}
\[
\begin{aligned}
&\text{err}_{Tay}\revision{\leq\tilde{C}}\|\mathcal{L}^{\ddot{s}}[\uld{\nlc}\,\aaa+\bbb\,\uld{u}]\|_{L^2(L^2)}\\
&\leq
\revision{\tilde{C}}\Bigl(\|\uld{\nlc}\,\mathcal{L}^{\ddot{s}}[\aaa]\|_{L^2(L^2)}
+\|\aaa\,\mathcal{L}^{\ddot{s}}[\uld{\nlc}]\|_{L^2(L^2)}
+C^d_{KP}\|\mathcal{L}^{s_1}[\uld{\nlc}]\|_{L^p}\|\mathcal{L}^{\ddot{s}-s_1}[\aaa]\|_{L^2(L^{p^\sharp})}\\
&\quad+\|\uld{u}\,\mathcal{L}^{\ddot{s}}[\bbb]\|_{L^2(L^2)}
+\|\bbb\,\mathcal{L}^{\ddot{s}}[\uld{u}]\|_{L^2(L^2)}
+C^d_{KP}\|\mathcal{L}^{\stil_1}[\uld{u}]\|_{L^q(L^\ptil)}\|\mathcal{L}^{\ddot{s}-\stil_1}[\bbb]\|_{L^{q^\sharp}(L^{\ptil^\sharp})}\Bigr)\\
&\leq
\revision{\tilde{C}}\Bigl(\|\uld{\nlc}\|_{L^{p_1}}\|\mathcal{L}^{\ddot{s}}[\aaa]\|_{L^2(L^{p_1^\sharp})}
+\|\aaa\|_{L^2(L^{p_2^\sharp})}\|\mathcal{L}^{\ddot{s}}[\uld{\nlc}]\|_{L^{p_2}}
+C^d_{KP}\|\mathcal{L}^{s_1}[\uld{\nlc}]\|_{L^p}\|\mathcal{L}^{\ddot{s}-s_1}[\aaa]\|_{L^2(L^{p^\sharp})}\\
&\quad+\|\uld{u}\|_{L^{q_1}(L^{\ptil_1})}\|\mathcal{L}^{\ddot{s}}[\bbb]\|_{L^{q_1^\sharp}(L^{\ptil_1^\sharp})}
+\|\bbb\|_{L^{q_2^\sharp}(L^{\ptil_2^\sharp})}\|\mathcal{L}^{\ddot{s}}[\uld{u}]\|_{L^{q_2}(L^{\ptil_2})}\\
&\hspace*{5cm}+C^d_{KP}\|\mathcal{L}^{\stil_1}[\uld{u}]\|_{L^q(L^\ptil)}\|\mathcal{L}^{\ddot{s}-\stil_1}[\bbb]\|_{L^{q^\sharp}(L^{\ptil^\sharp})}\Bigr)
\end{aligned}
\]  
\marginrr{3.p15-8}
for 
\begin{equation}\label{pp1p2}
p,\,p_1,\,p_2,\,\ptil,\,\ptil_1,\,\ptil_2,\,q,\,q_1,\,q_2\,\geq2, \quad s_1,\,\revision{\stil_1}\in[0,\ddot{s}]
\end{equation}
\marginrr{3.p16-1}
where we balance Sobolev strength of norms on $\uld{\nlc}$ and on $\uld{u}$, also matching the relevant $\XXX$ norm parts $H^s(\Omega)$ and $H^{\check{\sigma}}(0,T;H^{\check{s}}(\Omega))$ by setting
\begin{equation}\label{stil1sddot}
\begin{aligned}
&s_1=\ddot{s},\quad p_2=p, \quad s-\frac{d}{2}=\ddot{s}-\frac{d}{p}=-\frac{d}{p_1}, \quad\\
&\stil_1=\ddot{s},\quad \ptil_2=\ptil, \quad \check{s}-\frac{d}{2}=\ddot{s}-\frac{d}{\ptil}=-\frac{d}{\ptil_1}, \quad q_1=q_2=q, \quad \check{\sigma}-\frac12=-\frac{1}{q}.
\end{aligned}
\end{equation}
These requirements uniquely determine 
$p,\,p_1,\,p_2,\,\ptil,\,\ptil_1,\,\ptil_2,\,q,\,q_1,\,q_2,\,s_1,\,\stil$ 
for given $s$, $\ddot{s}$, \revision{$\check{s}$, $\check{\sigma}$} as
\marginrr{3.p16-2} 
\[
\begin{aligned}
&p_2=p=\bar{p}(d,s-\ddot{s}),\quad p_1=\bar{p}(d,s),\quad
s_1=\stil_1=\ddot{s},\\
&\ptil_2=\ptil=\bar{p}(d,\check{s}-\ddot{s}),\quad \ptil_1=\bar{p}(d,\check{s}),\quad
q_1=q_2=q=\bar{p}(1,\check{\sigma})
,\\
&\bar{p}(d,s)\begin{cases}
=\frac{2d}{d-2s}&\text{ if }s<\tfrac{d}{2}\\
<\infty&\text{ if }s=\tfrac{d}{2}\\
=\infty&\text{ if }s>\tfrac{d}{2}
\end{cases}
\end{aligned}
\]  
Then \eqref{cond_ssigma} with $\ddot{\sigma}=0$ implies \eqref{pp1p2}.
Altogether these choices also balance the norms of $\aaa$, $\bbb$ and by Sobolev embeddings lead to  
\begin{equation}\label{estTay}
\begin{aligned}
\text{err}_{Tay}
&\leq \revision{\tilde{C}}\,C_{Sob}(1+C^d_{KP}) \|(\uld{\nlc},\uld{u})\|_{\XXX}
\Bigl(\|\mathcal{L}^{\ddot{s}}[\aaa]\|_{L^2(L^{p_1^\sharp})} 
+\|\mathcal{L}^{\ddot{s}}[\bbb]\|_{L^{q_1^\sharp}(L^{\ptil_1^\sharp})}\Bigr)\\
&= \revision{\tilde{C}}\,C_{Sob}(1+C^d_{KP}) \|(\uld{\nlc},\uld{u})\|_{\XXX}
\Bigl(\|\mathcal{L}^{\ddot{s}}[\aaa]\|_{L^2(L^{d/s})} 
+\|\mathcal{L}^{\ddot{s}}[\bbb]\|_{L^{1/\check{\sigma}}(L^{d/\check{s}})}\Bigr)\revision{.}
\end{aligned}
\end{equation}
\marginR{7.}

An analogous estimate can easily be obtained in case $\ddot{\sigma}>\frac12$, in which $H^{\ddot{\sigma}}(0,T)$ is a Banach algebra, so that we can simply replace $L^2(0,T)$, $L^{q}(0,T)$, $L^{q^\sharp}(0,T)$, $L^{q_1}(0,T)$, $L^{q_1^\sharp}(0,T)$, $L^{q_2}(0,T)$, $L^{q_2^\sharp}(0,T)$ by $H^{\ddot{\sigma}}(0,T)$ to obtain
\begin{equation}\label{estTay1}
\begin{aligned}
\text{err}_{Tay}
&\leq \revision{\tilde{C}}\,C_{Sob}(1+C^d_{KP}) \|(\uld{\nlc},\uld{u})\|_{\XXX}
\Bigl(\|\mathcal{L}^{\ddot{s}}[\aaa]\|_{H^{\ddot{\sigma}}(L^{p_1^\sharp})} 
+\|\mathcal{L}^{\ddot{s}}[\bbb]\|_{H^{\ddot{\sigma}}(L^{\ptil_1^\sharp})}\Bigr)\\
&= \revision{\tilde{C}}\,C_{Sob}(1+C^d_{KP}) \|(\uld{\nlc},\uld{u})\|_{\XXX}
\Bigl(\|\mathcal{L}^{\ddot{s}}[\aaa]\|_{H^{\ddot{\sigma}}(L^{d/s})} 
+\|\mathcal{L}^{\ddot{s}}[\bbb]\|_{H^{\ddot{\sigma}}(L^{d/\check{s}})}\Bigr)\revision{.}
\end{aligned}
\end{equation}
\marginR{7.}

\begin{corollary}\label{cor:stabSobolev}
With a reference pressure $\revision{\hat{u}^0_m}(x)=\phi(x)\,\revision{\frac{2}{T}}\int_0^T \psi(t)\, e^{-\imath m\omega t}\, dt$, whose time dependent factor $\psi$ with $\widetilde{\mathfrak{B}}(o)=\revision{\frac{2}{T}}\int_0^T \psi^2(t) e^{-ot}\, dt$ 
\marginrr{3.p16-3} 
satisfies 
\begin{equation}\label{kappas}
\lambda_j^{\kappa_1}
\leq M_1\,|\widetilde{\mathfrak{B}}({\pole_j})|\,e^{\Re({\pole_j})T}
\begin{cases}
\sqrt{-\Re({\pole_j})}&\text{ if }\ddot{\sigma}=0\\
|{\pole_j}|&\text{ if }\ddot{\sigma}=1 
\end{cases}
, \quad
\widetilde{\mathfrak{B}}(\imath m\omega)
\leq M_2 (m\omega)^{\kappa_2}
\end{equation}
\revision{and whose space dependent factor satisfies \eqref{phiLinfty}}
\marginrr{2.3.1.(a)}
on the spaces 
\[
\XXX=H^s(\Omega)\times h^{\check{\sigma}}(H^{\check{s}}(\Omega)), \quad 
\YYY= h^{\ddot{\sigma}}(H^{\ddot{s}}(\Omega))\times \YYY^{obs}_{pol}
\]
with $s$, $\check{s}$, $\ddot{s}$, $\check{\sigma}$, $\ddot{\sigma}$ satisfying
\begin{equation}\label{condssigma}
\frac12<s, \quad 0\leq\ddot{s}\leq\check{s}\leq s
,\quad
\check{\sigma}=\ddot{\sigma}\in\{0,1\}, \quad
\kappa_2+\frac12<s-\check{s}\leq s-\ddot{s}\leq2\kappa_1,
\end{equation}  
there exist $\rho>0$, $C>0$ such that for all $(\nlc,\tilde{\uvec})$ such that  
\[
\|(\tilde{u}-u^0)\,(\tilde{u}+u^0)\|_{H^{\ddot{\sigma}}(0,T;W^{\ddot{s},d/s}(\Omega))} 
+\|\nlc\,(\tilde{u}+u)\|_{Z(0,T;W^{\ddot{s},d/\check{s}}(\Omega))}
\leq\rho,\quad
Z=\begin{cases}
L^\infty&\text{if }\ddot{\sigma}=0\\
H^{\ddot{\sigma}}&\text{if }\ddot{\sigma}=1
\end{cases}
\]
the stability and uniqueness estimates
\eqref{stabest_nlcu1}, \eqref{stabest_nlcu_ext} hold.
\end{corollary}
\begin{proof}
The conditions \eqref{condssigma} are a collection of 
the conditions 
\eqref{cond_ssigma}, 
$s>1/2$ according to Lemma~\ref{lem:trace2coeff}, 
as well as \eqref{amplfac_lambda}, summability and boundedness in \eqref{constants}, cf. \eqref{factors_JMGT},
under the given asymptotics of $\widetilde{\mathfrak{B}}_m$ and $\widetilde{\mathfrak{B}}(\pole_j)$
\[
s-\ddot{s}\leq2\kappa_1,\quad 
\check{s}\leq s,\quad
\check{\sigma}\leq\ddot{\sigma},\quad
\check{\sigma}+\check{s}-s+\kappa_2<-\frac12.
\]
\end{proof}
\begin{remark}
If 
\begin{equation}\label{necessarykappa}
\max\{0,\kappa_2\}+\frac12<2\kappa_1
\end{equation}
then a possible choice is 
\begin{equation}\label{choices}
s\in(\max\{0,\kappa_2\}+\frac12,2\kappa_1], \quad \ddot{s}=\check{s}=0, \quad \ddot{\sigma}=\check{\sigma}\in\{0,1\}
\end{equation}
Using an all-at once formulation of the inverse problem, we are not bound to using typical PDE solution spaces; we use this freedom for choosing the indices $s$, $\check{s}$, $\ddot{s}$, $\check{\sigma}$, $\ddot{\sigma}$ in Corollary~\ref{cor:stabSobolev}.
\end{remark}
The findings of Section~\ref{sec:discussion} below 
\revision{
lead us to the choice 
\[s\in(\tfrac12,1], \quad \ddot{s}=\check{s}=0,  \quad \ddot{\sigma}\textcolor{blue}{=\check{\sigma}}=1,\]
and} 
allow us to draw the 
\revision{following} 
conclusion.
\begin{corollary}\label{cor:stabJMGT}
With $\tau>0$, a reference pressure $u^0_m(x)=\phi(x)\,\revision{\frac{2}{T}}\int_0^T \psi(t)\, e^{-\imath m\omega t}\, dt$, whose time dependent factor $\psi$ with $\widetilde{\mathfrak{B}}(o)=\revision{\frac{2}{T}}\int_0^T \psi^2(t) e^{-ot}\, dt$ is chosen according to Section~\ref{rem:B_m}, 
\revision{and whose space dependent factor satisfies \eqref{phiLinfty}, as well as the spaces
\marginrr{2.3.1.(a)}
\[
\XXX=H^s(\Omega)\times h^1(L^2(\Omega)), \quad 
\YYY= h^1(L^2(\Omega))\times \YYY^{obs}_{pol}
\]
with $s\in(\frac12,1]$,
}
there exist $\rho>0$, $C>0$ such that for all $(\nlc,\tilde{\uvec})$ such that  
\[
\textcolor{blue}{\|(\tilde{u}-u^0)\,(\tilde{u}+u^0)\|_{H^1(0,T;L^{d/s}(\Omega))} 
+\|\nlc\,(\tilde{u}+u)\|_{H^1(0,T;L^{\infty}(\Omega))}
}
\leq\rho,
\]
the stability and uniqueness estimates
\eqref{stabest_nlcu1}, \eqref{stabest_nlcu_ext} hold.
\end{corollary}
We emphasize that a positive relaxation time $\tau>0$ is essential for this result, since the question on how to choose the source such that the factor $\check{C}^{(\check{s}-s,\check{\sigma})}_{\vartheta\Theta}$ is finite remains open in case $\tau=0$, that is, for the classical Westervelt model, cf. \eqref{factors_West} in Section~\ref{rem:different_models} below.

\subsection{Discussion of conditions}\label{sec:discussion}

\subsubsection{Noise amplification} \label{rem:amplfac}
The amplification factor of the observation contribution $\uld{p}^{obs}$ for the nonlinearity coefficient increment $\uld{\nlc}$ is critical for an assessment of the degree of ill-posedness of the inverse problem, since it relates the imaging quantity to the part of the data, that is contaminated by measurement noise (while the PDE residual part $\uld{r}$ is given exactly).  
According to \eqref{constants}, \eqref{stabest_nlc}, \eqref{Ynorms_pol}, it consists of
\begin{itemize}
\item factors $\|(\text{Tr}_\Sigma^{s,j})^{-1}\|_{L(H^{s-1/2}(\Sigma))\to h^{s,j}}$ depending on the geometry of $\Sigma$ cf. Lemma~\ref{lem:trace2coeff};
\item factors $\frac{1}{\widetilde{\mathfrak{B}}({\pole_j})}$ depending on the choice of the excitation (via $\mathfrak{B}_m:=\mathcal{B}_m(\hat{\psi},\hat{\psi})$) and on the PDE (via location of the poles ${\pole_j}$);
\item a factor 
$\bar{C}_{\vartheta\Theta}
=\sup_{\ell\in\mathbb{N}} \left|\frac{\Theta({\pole_\ell})\Psi'({\pole_\ell})}{\pole_\ell^2}\right|
=\sup_{\ell\in\mathbb{N}} \left|\frac{\vartheta'({\pole_\ell})+\lambda_\ell\Theta'({\pole_\ell})}{\pole_\ell^2}\right|$
depending on the PDE via the functions $\vartheta$, $\Theta$. 
\end{itemize} 

\subsubsection{Pole location} \label{rem:loc_poles}
The poles $\pole_j=\tfrac{1}{z_j}$ depend on the PDE model as they are zeros of the function $o\mapsto \vartheta(o)+\Theta(o)\lambda_j$.
They appear in several of the contributions to the noise amplification factor (cf. Section~\ref{rem:amplfac}).
Additionally, the estimate \eqref{est_interpol_r}, \eqref{est_interpol_r1} and the condition \eqref{amplfac_lambda} explicitly show the relevance of the growth of the negative real part of the poles for the stability of the inverse problem. 
Ideally, if the sequence 
$(-\Re(\pole_j))_{j\in\mathbb{N}}$ 
stays bounded or grows only logarithmically, then these factors stay bounded or only grow polynomially.
It is intuitive that the closer the poles are to the imaginiary axis where measurements $\uld{p}^{obs}_m$ and PDE data $\uld{r}_m$ are taken, the better posed the inverse problem is.
From a physical point of view, attenuation is responsible for the loss of information and indeed,  
it is the attenuation coefficient $\delta$ that moves the poles away from the imaginary axis.
As we will see below, the relaxation time $\tau$ counteracts this effect.
This is illustrated in Figure~\ref{fig:poles} and can also be seen from the following sensitivity analysis.
Using polar coordinates $\pole_j=\radi_j\cos(\argu_j)+\imath\sin(\argu_j)$ the equation for the poles reads as (skipping the subscript $j$)
\[
f(\radi,\argu;\delta,\tau)=\left(\begin{array}{l}
f^\Re(\radi,\argu;\delta,\tau)\\
f^\Im(\radi,\argu;\delta,\tau)
\end{array}\right)=\left(\begin{array}{l}
\tau\radi^3\cos(3\argu)+\radi^2\cos(2\argu)+(\tau c^2+\delta)\lambda\radi\cos(\argu)+c^2\lambda\\
\tau\radi^3\sin(3\argu)+\radi^2\sin(2\argu)+(\tau c^2+\delta)\lambda\radi\sin(\argu)
\end{array}\right).
\]
We apply the Implicit Function Theorem 
\[
\frac{\partial (\radi,\argu)}{\partial(\delta,\tau)}=-\frac{\partial f}{\partial(\radi,\argu)}^{-1}\frac{\partial f}{\partial(\delta,\tau)}
=-\left(\begin{array}{cc} 
f^\Re_\radi&f^\Re_\argu\\
f^\Im_\radi&f^\Im_\argu
\end{array}\right)^{-1}
\left(\begin{array}{cc} 
f^\Re_\delta&f^\Re_\tau\\
f^\Im_\delta&f^\Im_\tau
\end{array}\right)
\]
to assess the sensitivity of the argument $\theta$ with respect to the parameters $\delta$ and $\tau$. To this end, note that according to \cite[Lemma 5.1]{nonlinearity_imaging_fracWest}, the poles have negative real parts and therefore $\argu\in[\tfrac{\pi}{2},\tfrac{3\pi}{2}]$.
The derivatives compute as
\[\begin{aligned}
f^\Re_\radi&=3\tau\radi^2\cos(3\argu)+2\radi\cos(2\argu)+(\tau c^2+\delta)\lambda\cos(\argu)\\
f^\Im_\radi&=3\tau\radi^2\sin(3\argu)+2\radi\sin(2\argu)+(\tau c^2+\delta)\lambda\sin(\argu)\\
f^\Re_\argu&=-\radi f^\Im_\radi\\
f^\Im_\argu&=\radi f^\Re_\radi\\
f^\Re_\delta&=\lambda\radi\cos(\argu)\\
f^\Im_\delta&=\lambda\radi\sin(\argu)\\
f^\Re_\tau&=\radi^3\cos(3\argu)+c^2\lambda\radi\cos(\argu)\\
f^\Im_\tau&=\radi^3\sin(3\argu)+c^2\lambda\radi\sin(\argu)\\
\end{aligned}\]
and Cramer's rule as well as trigonometric identities yield
\[\begin{aligned}
\frac{\partial \argu}{\partial\delta}&=
-\frac{f^\Re_\radi f^\Im_\delta-f^\Im_\radi f^\Re_\delta}{f^\Re_\radi f^\Im_\argu-f^\Im_\radi f^\Re_\argu}=\frac{\lambda\radi}{(f^\Re_\radi)^2+(f^\Im_\radi)^2}
2(3\tau\radi\cos(\argu)+1)\sin(\argu)
\\
\frac{\partial \argu}{\partial\tau}&=
-\frac{f^\Re_\radi f^\Im_\tau-f^\Im_\radi f^\Re_\tau}{f^\Re_\radi f^\Im_\argu-f^\Im_\radi f^\Re_\argu}=\frac{\radi}{(f^\Re_\radi)^2+(f^\Im_\radi)^2}
2\bigl((2\tau c^2-\delta)\lambda\radi\cos(\argu)+(c^2\lambda-\radi^2)\bigr)\sin(\argu).
\end{aligned}\]
Thus the change of the negative real part of the pole is governed by 
\[\begin{aligned}
\frac{\partial}{\partial\delta}(-\Re(\pole))
&=\frac{\partial}{\partial\delta}(-\radi\cos(\argu))=\radi\sin(\argu)\frac{\partial\argu}{\partial\delta}\\
&=\frac{2\lambda\radi^2\sin^2(\argu)}{(f^\Re_\radi)^2+(f^\Im_\radi)^2}(3\tau\Re(\pole)+1)
\end{aligned}\]
which is positive (that is, ill-posedness increases due to increasing $\delta$) iff
\begin{equation}\label{partialminusRepolepartialdelta}
\frac{\partial}{\partial\delta}(-\Re(\pole))>0
\quad \Leftrightarrow \quad -\Re(\pole)<\frac{1}{3\tau}.
\end{equation}
In particular, in case $\tau=0$, an increase of $\delta$ always causes an increase of the 
\revision{ill-posedess}, 
\marginR{8.}
whereas $\tau>0$ according to \eqref{partialminusRepolepartialdelta} limits this effect.

Analogously, we investigate the influence of increasing $\tau$ on the magnitude of the negative real part of the pole for fixed $\delta>0$.
We have  
\[\begin{aligned}
\frac{\partial}{\partial\tau}(-\Re(\pole))
&=\frac{\partial}{\partial\tau}(-\radi\cos(\argu))=\radi\sin(\argu)\frac{\partial\argu}{\partial\tau}\\
&=\frac{2\radi^2\sin^2(\argu)}{(f^\Re_\radi)^2+(f^\Im_\radi)^2}
\bigl((2\tau c^2-\delta)\lambda\radi\cos(\argu)+(c^2\lambda-\radi^2)\bigr),
\end{aligned}\]
which with $\radi\cos(\argu)=\Re(\pole)$, $\radi^2=|\pole|^2$ yields the expected decrease of ill-posedness with increasing $\tau$ iff
\begin{equation}\label{partialminusRepolepartialtau}
\frac{\partial}{\partial\tau}(-\Re(\pole))<0
\quad \Leftrightarrow \quad\begin{cases}
 -\Re(\pole)<\frac{|\pole|^2-c^2\lambda}{\delta-2\tau c^2}\text{ and }\delta>2\tau c^2\text{ or }\\
 -\Re(\pole)>\frac{c^2\lambda-|\pole|^2}{2\tau c^2-\delta}\text{ and }\delta<2\tau c^2\text{ or }\\
c^2\lambda<|\pole|^2\text{ and }\delta=2\tau c^2
\end{cases}
\end{equation}
where according to \eqref{zell_JMGT} we have 
$|\,|\pole|^2-c^2\lambda|=O(\sqrt{\lambda})=O(|\pole|)$
as $\lambda\to\infty$. 
More precisely, using the depressed cubic equation for $\radi=|\pole|$ that results from eliminating the quadratic term 
\[
0=f^\Re(\radi,\argu;\delta,\tau)\sin(2\argu)-f^\Im(\radi,\argu;\delta,\tau)\cos(2\argu)
=\sin(\argu)\bigl(\tau\radi^3-(\tau c^2+\delta)\lambda\radi+2c^2\lambda\cos(\argu)\bigr)
\]
and applying Vi\`{e}te's triginometric formula for the roots of a cubic equation, we obtain
\[\begin{aligned}
r=& \tfrac{2}{\sqrt{3}}\, \tilde{c}\, \sqrt{\lambda}\, \cos\bigl(\tfrac13 \textrm{arccos}(-\cos(\argu)\tfrac{c^2}{\tau c^2+\delta}\tfrac{\sqrt{3}}{\tilde{c}\, \sqrt{\lambda}}\bigr)
\end{aligned}\] 
(the other two roots don't qualify because of their sign or their asymptotic magnitude as $\lambda\to\infty$)
where we recall $\tilde{c}^2=c^2+\frac{\delta}{\tau}$ and define $\tilde{b}=\tau c^2+2\delta+\frac{\delta^2}{\tau c^2}$.
Due to the fact that $\textrm{arccos}(0)=\frac{\pi}{2}$, $-\cos(\argu)=-\Re(\pole)/\radi>0$, $\cos(\frac{\pi}{6})=\frac{\sqrt{3}}{2}$, $\textrm{arccos}'(0)=-1$, and $\cos'(\frac{\pi}{6})=-\frac12$, we thus have
\[
|\pole|=\radi\in \bigl(\tilde{c}\, \sqrt{\lambda},\, \tilde{c}\, \sqrt{\lambda} (1+O(\tfrac{1}{\tilde{b}\sqrt{\lambda}}))\bigr) 
\] 
As a consequence, $|\pole|^2>c^2\lambda$ always holds and from \eqref{partialminusRepolepartialtau} we obtain
\[
\frac{\partial}{\partial\tau}(-\Re(\pole))<0
\quad \Leftarrow \quad \delta\leq 2\tau c^2.
\]
That is, as soon as $\tau$ lies above the threshold $\frac{\delta}{2c^2}$, its growth decreases ill-posedeness.

\begin{figure}
\includegraphics[width=0.3\textwidth]{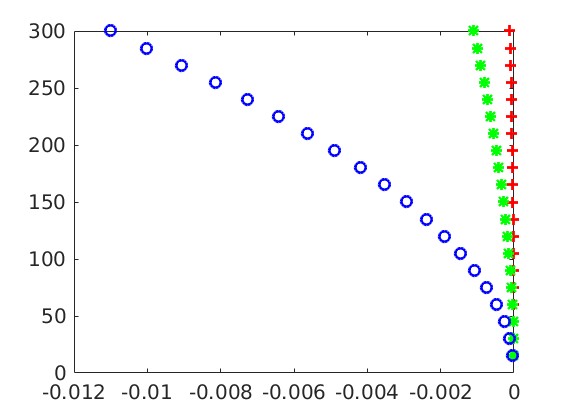}
\hspace*{0.02\textwidth}
\includegraphics[width=0.3\textwidth]{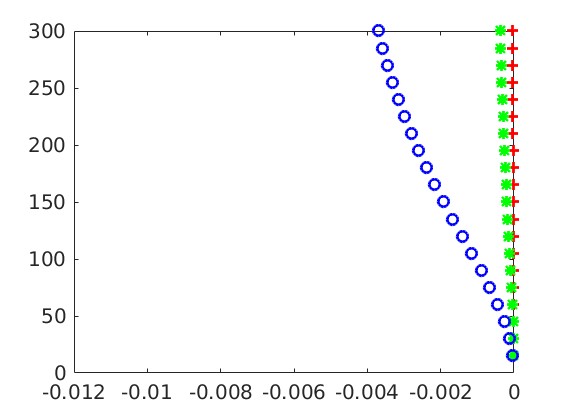}
\hspace*{0.02\textwidth}
\includegraphics[width=0.3\textwidth]{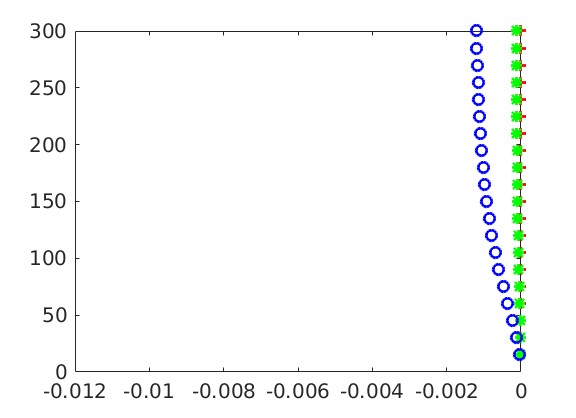}
\caption{Poles at different levels of attenuation 
($\delta=6*10^{-9}$\ldots\textcolor{red}{
\textbf{+}
}, 
$\delta=6*10^{-8}$\ldots\textcolor{green}{
\textbf{*}
}, 
$\delta=6*10^{-7}$\ldots\textcolor{blue}{
$\mathbf{\circ}$})
and with different relaxation times ($\tau=0.001\mu s$\ldots left, $\tau=0.005\mu s$\ldots middle, $\tau=0.01\mu s$\ldots right); $c=1500 m/s$, $\Omega=(0,10 mm)$
\label{fig:poles}}
\end{figure}

\subsubsection{PDE dependent image space norms} \label{rem:constants}
Besides $\bar{C}_{\vartheta\Theta}$, also the other constants appearing in the norms, cf. \eqref{constants}, depend on the PDE. 
Solving the elementary minimization problems for the denominator in the last two constants in \eqref{constants} with respect to $\lambda_j>\lambda_1$, 
\[
\begin{aligned}
\inf_{\lambda\geq\lambda_1}\lambda^\alpha|\vartheta(o_m)+\Theta(o_m)\lambda|^2
&\geq\min_{\lambda\geq\lambda_1}\lambda^\alpha\bigl(|\vartheta(o_m)|^2+2\Re(\vartheta(o_m)\overline{\Theta(o_m)})\lambda+|\Theta(o_m)|^2\lambda^2\bigr)\\
&=\min\{j(\lambda_1),j(\lambda_-),j(\lambda_+)\},
\end{aligned}
\]
where with $a^2= |\Theta(o_m)|^2$, $b=-\Re(\vartheta(o_m)\overline{\Theta(o_m)})$, $d^2=|\vartheta(o_m)|^2$, thus $a^2d^2\geq b^2$, 
\[
\begin{aligned}
&j(\lambda)=a^2\lambda^{\alpha+2}-2b\lambda^{\alpha+1}+d^2\lambda^{\alpha}
= 
\frac{\lambda^\alpha}{\alpha+2}\Bigl((\alpha+2)a^2\lambda^2-2(\alpha+1)b\lambda+\alpha d^2\ 
-2b\lambda+2d^2\Bigr)\\
&j'(\lambda)=\lambda^{\alpha-1}((\alpha+2)a^2\lambda^2-2(\alpha+1)b\lambda+\alpha d^2)\\
&\lambda_\pm=\frac{(\alpha+1)b\pm\sqrt{D}}{(\alpha+2)a^2},
\qquad
D=(\alpha+1)^2b^2-\alpha(\alpha+2)a^2d^2=a^2d^2-(\alpha+1)^2(a^2d^2-b^2)\\
&j(\lambda_\pm)=\frac{2\lambda_\pm^\alpha}{\alpha+2}(d^2-b\lambda_\pm)
\end{aligned}
\]
Considering both cases $\tau>0$ (JMGT) and $\tau=0$ (the classical Westervelt model), we have (with $o_m^2=-|o_m|^2$)
\[\begin{aligned}
a^2=(\tau c^2+\delta)^2 |o_m|^2+c^4,  \quad d^2=\tau^2|o_m|^6+|o_m|^4, \quad   
\\
b=\tau(\tau c^2+\delta)|o_m|^4+c^2|o_m|^2,  \quad  
a^2d^2-b^2
=\delta^2 |o_m|^6,
\end{aligned}\] 
which asymptotically (that is, for large $|o_m|$) yields
$\sqrt{D}\sim b\sim \tau(\tau c^2+\delta)|o_m|^4$, so that
\[\begin{aligned}
&j(\lambda_+)
\sim \frac{2}{\alpha+2}\left(\frac{b}{a^2}\right)^\alpha\frac{a^2d^2-b^2}{a^2}
\sim \frac{2}{\alpha+2}\frac{\tau^\alpha\delta^2}{(\tau c^2+\delta)^{2+\alpha}}\, |o_m|^{4+2\alpha}\\
&j(\lambda_-)
\sim \frac{4\alpha^\alpha}{(\alpha+2)^{\alpha+2}}\left(\frac{b}{a^2}\right)^\alpha\frac{2a^2d^2+\alpha(a^2d^2-b^2)}{a^2}
\sim \frac{4\alpha^\alpha}{(\alpha+2)^{\alpha+2}}
\frac{\tau^{\alpha+2}}{(\tau c^2+\delta)^\alpha}\, |o_m|^{6+2\alpha}
\end{aligned}\] 
hence 
\begin{equation}\label{min_denominator}
\begin{aligned}
&\inf_{\lambda\geq\lambda_1}|\vartheta(o_m)+\Theta(o_m)\lambda|^2
=\frac{\delta^2\,|o_m|^6}{c^4+(\tau c^2+\delta)^2\,|o_m|^2} 
\geq \frac{\delta^2}{c^4/\omega^2+(\tau c^2+\delta)^2} |o_m|^4\\
&\inf_{\lambda\geq\lambda_1}\lambda^\alpha|\vartheta(o_m)+\Theta(o_m)\lambda|^2
\sim 
\begin{cases} 
\frac{2}{\alpha+2}
\frac{\tau^\alpha\delta^2}{(\tau c^2+\delta)^{2+\alpha}}\, |o_m|^{4+2\alpha} 
&\text{ if }\tau>0\\
\lambda_1^{\alpha}|o_m|^4 &\text{ if }\tau=0
\end{cases}
\end{aligned}
\end{equation}
(since in case $\tau=0$ the discriminant $D$ is negative for $\alpha>0$ and large $|o_m|$, hence the minimum is attained at $\lambda=\lambda_1$)
and therefore
\begin{equation}\label{constants_JMGT}
\begin{aligned}
&\check{C}^{(-\alpha,\check{\sigma})}_{\vartheta\Theta}
\lesssim
\begin{cases}
\left(\sum_{m\in\N}|o_m^{\check{\sigma}-\alpha}\,\mathfrak{B}_m|^2\right)^{1/2}& \text{ if }\tau>0\\
\left(\sum_{m\in\N}|o_m^{\check{\sigma}}\mathfrak{B}_m|^2\right)^{1/2}& \text{ if }\tau=0
\end{cases}\\
&\hat{C}_{\vartheta\Theta}^{(0,\check{\sigma}-\ddot{\sigma})}
\leq\frac{\sqrt{c^4/\omega^2+(\tau c^2+\delta)^2}}{\delta}
 \sup_{m\in\N} |o_m^{\check{\sigma}-\ddot{\sigma}}|
\leq\frac{\sqrt{c^4/\omega^2+(\tau c^2+\delta)^2}}{2\tau c^2+\delta}
 \,|o_1^{\check{\sigma}-\ddot{\sigma}}|
\end{aligned}
\end{equation}
Note that as opposed to plain Westervelt ($\tau=0$), JMGT provides the chance to achieve summability in $\check{C}^{(-\alpha,\check{\sigma})}_{\vartheta\Theta}$ by choosing $\alpha=s-\check{s}>0$ large enough, even for just bounded or even (polynomially) increasing factors $|\mathfrak{B}_m|$, cf. \eqref{BmdeltaT}.
In view of the requirement $\check{s}-\ddot{s}\geq0$ from \eqref{cond_ssigma}, we restrict ourselves to considering the best possible case of $\hat{C}_{\vartheta\Theta}^{(\check{s}-\ddot{s},\check{\sigma}-\ddot{\sigma})}$ here.

\subsubsection{Comparison of models}\label{rem:different_models}
We now compare the items discussed in Sections~\ref{rem:amplfac}, \ref{rem:loc_poles}, \ref{rem:constants}, for some models relevant for the underlying nonlinearity imaging problem.
The uniqueness proof given in Sections~\ref{sec:uniqueness-lin}, \ref{sec:uniqueness-nonlin} is sufficiently general to comprise a large variety of models via  different choices of $\vartheta$, $\Theta$. 
In fact, JMGT via limiting cases of vanishing relaxation time $\tau$ and/or diffusivity of sound $\delta$ contains some of them.

For two of these models (Westervelt, JMGT with strong damping), existence of periodic solutions has already been established \cite{periodicJMGT,periodicWestervelt,periodicWest_2}; for the others this is still missing. Note, however, that with the all-at-once approach we are using here, well-posedness of the forward problem is not a strict requirement, rather a strong motivation.

We start by recalling the PDE dependent items in the uniqueness proof for the JMGT case 
$\tau>0$, cf. \eqref{constants_JMGT}.
From \eqref{zell_JMGT}, \eqref{amplfac_JMGT}, \eqref{constants_JMGT} and \cite[Lemma 5.1]{nonlinearity_imaging_fracWest}, we have 
\begin{equation}\label{factors_JMGT}
\begin{aligned}
&|\pole_\ell|= \tilde{c}\sqrt{\lambda_\ell} +O(1)+\mathcal{O}(\omzer^{1/2}\lambda_\ell^{-1/4})+\mathcal{O}(\delta^{1/2}\lambda_\ell^{1/4})\text{ with }\tilde{c}^2=c^2+\frac{\delta}{\tau}\\
&0\leq-\Re({\pole_\ell})=\mathcal{O}(\omzer^{1/2}\lambda_\ell^{-1/4})+\mathcal{O}(\delta^{1/2}\lambda_\ell^{1/4})
\\
&\bar{C}_{\vartheta\Theta} 
\lesssim 2\tau\\
&\check{C}^{(\check{s}-s,\check{\sigma})}_{\vartheta\Theta}
\lesssim
\left(\sum_{m\in\N}|o_m^{\check{\sigma}+\check{s}-s}\,\mathfrak{B}_m|^2\right)^{1/2}\\
&\hat{C}_{\vartheta\Theta}^{(0,\check{\sigma}-\ddot{\sigma})}
\leq\frac{\sqrt{c^4/\omega^2+(\tau c^2+\delta)^2}}{2\tau c^2+\delta}
 \,|o_1^{\check{\sigma}-\ddot{\sigma}}|
\end{aligned}
\end{equation}
\marginR{9.}
for $s\geq\check{s}$, $\ddot{\sigma}\geq\check{\sigma}$.
The behaviour of the poles heavily depends on the strength of the damping.
In the strong damping case $\delta>0$, the negative real parts of the poles tend to infinity with a rate of $\lambda_\ell^{1/4}$, whereas with weak or no damping $\delta=0$, $\omzer\geq0$ they decrease.

We first of all compare \eqref{factors_JMGT} for JMGT to the pure undamped second order wave equation case $\tau=\delta=\omzer=0$, where 
\begin{equation*}
\begin{aligned}
&\Psi(o)=-\frac{o^2}{c^2}, \quad 
\Psi'(o)=-2\frac{o}{c^2},\quad
{\pole_\ell}=\imath c \sqrt{\lambda_\ell},\quad
\frac{\Theta({\pole_\ell})\Psi'({\pole_\ell})}{\pole_\ell^2}
=-2z_\ell=\imath \tfrac{2}{c} \lambda_\ell^{-1/2}.
\end{aligned}
\end{equation*}
Thus, both the location of the poles is ideal, namely on the imaginary axis, and the amplification factor $\frac{\Theta({\pole_\ell})\Psi'({\pole_\ell})}{\pole_\ell^2}$ decays with a rate $\lambda_j^{-1/2}$, 
which indicates a (theoretical) one derivative gain of reconstruction under the undamped wave equation model as compared to reconstruction under the JMGT model.
However, \eqref{min_denominator} doesn't yield finiteness of 
$\check{C}^{(0,\check{\sigma})}_{\vartheta\Theta}$ and $\hat{C}_{\vartheta\Theta}^{(0,\check{\sigma}-\ddot{\sigma})}$. To bound these quantities, one would rather have to more closely inspect the location of the poles on the imaginary axis and their distances from the measurement points $\imath m\omega$.

Including weak damping in the second order wave equation $\tau=0$, $\delta=0$, $\omzer>0$
\[
\vartheta(o)=o^2+\omzer o, \quad \Theta(o)=c^2,
\]
we have (with $o_m^2=-|o_m|^2$, $a^2=c^4$, $b=c^2|o_m|^2$, $d^2=|o_m|^4+\omzer|o_m|^2$, $a^2d^2-b^2=c^4\omzer|o_m|^2$, $D=c^4(|o_m|^4-\alpha(\alpha+2)\omzer|o_m|^2)$, $\sqrt{D}\sim c^2|o_m|^2=b$)
\[
\begin{aligned}
&\Psi(o)=-\frac{o^2+\omzer o}{c^2}, \quad 
\Psi'(o)=-\frac{2o+\omzer }{c^2},\\
&{\pole_\ell}=-\tfrac{\omzer }{2}\pm\imath\sqrt{c^2\lambda_\ell-\tfrac{\omega_o^2}{4}},\\
&|\tfrac{\Theta({\pole_\ell})\Psi'({\pole_\ell})}{\pole_\ell^2}|
=|- (\omzer z_\ell+2z_\ell^2)|= \frac{\omzer }{c}\,\lambda_\ell^{-1/2} (1+\mathcal{O}(\lambda_\ell^{-1}))\\
&\min_{\lambda\geq\lambda_1}|\vartheta(o_m)+\Theta(o_m)\lambda|^2
=\omzer ^2|o_m|^2\\
&\min_{\lambda\geq\lambda_1}\lambda^\alpha|\vartheta(o_m)+\Theta(o_m)\lambda|^2
\sim \tfrac{2}{\alpha+2} \tfrac{\omzer}{c^{2\alpha}}\, |o_m|^{2+2\alpha}
\end{aligned}
\]
and thus 
\begin{equation}\label{factors_West_weak}
\begin{aligned}
&|\pole_\ell|=c \sqrt{\lambda_\ell}\\
&-\Re({\pole_\ell})=\frac{\omzer }{2}
\\
&\bar{C}_{\vartheta\Theta}\geq 
\frac{\omzer }{c}\,\lambda_\ell^{-1/2} (1+\mathcal{O}(\lambda_\ell^{-1}))\\
&\check{C}^{(\check{s}-s,\check{\sigma})}_{\vartheta\Theta}
\lesssim
\left(\sum_{m\in\N}
|o_m^{1+\check{\sigma}+\check{s}-s}\mathfrak{B}_m|^2
\right)^{1/2}\\
&\hat{C}_{\vartheta\Theta}^{(0,\check{\sigma}-\ddot{\sigma})}
\leq
\tfrac{1}{\omzer ^2} \sup_{m\in\N} |o_m^{1+\check{\sigma}-\ddot{\sigma}}|.
\end{aligned}
\end{equation}
While weak damping does not significantly change the favorable location of the poles as compared to the undamped wave equation, in view of $\check{\sigma}\geq\ddot{\sigma}$ according to \eqref{cond_ssigma}, \eqref{Bmpsi}, still no finite bound 
$\hat{C}_{\vartheta\Theta}^{(\check{s}-\ddot{s},\check{\sigma}-\ddot{\sigma})}$ can be found here. 
Also note that solutions of the nonlinear version of \eqref{JMGT-Westervelt} with $\tau=\delta=0$ cannot be established, 
not even under initial instead of periodicity conditions.

Finally, the classical Westervelt case $\tau=0$, with strong damping $\delta>0$ leads to  
\[
\begin{aligned}
&\Psi(o)=-\frac{o}{\delta}(1+\mathcal{O}(\tfrac{1}{o})), \quad 
\Psi'(o)=-\frac{1}{\delta}(1+\mathcal{O}(\tfrac{1}{o})),\\
&{\pole_\ell}=-\delta\lambda_\ell(1+\mathcal{O}(\tfrac{1}{\lambda_\ell})),\\
&\tfrac{\Theta({\pole_\ell})\Psi'({\pole_\ell})}{\pole_\ell^2}
=- z_\ell=\tfrac{1}{\delta} \lambda_\ell^{-1}(1+\mathcal{O}(\tfrac{1}{\lambda_\ell}))
\end{aligned}
\]
and thus we obtain (cf. \eqref{constants_JMGT})
\begin{equation}\label{factors_West}
\begin{aligned}
&|\pole_\ell|=\delta\lambda_\ell+\mathcal{O}(1)\\
&-\Re({\pole_\ell})=\delta\lambda_\ell+\mathcal{O}(1)
\\
&\bar{C}_{\vartheta\Theta}\geq 
|\tfrac{1}{\delta} \lambda_\ell^{-1}(1+\mathcal{O}(\lambda_\ell^{-1}))|\\
&\check{C}^{(\check{s}-s,\check{\sigma})}_{\vartheta\Theta}
\lesssim\left(\sum_{m\in\N}|o_m^{\check{\sigma}}\mathfrak{B}_m|^2\right)^{1/2}\\
&\hat{C}_{\vartheta\Theta}^{(0,\check{\sigma}-\ddot{\sigma})}
\leq\frac{\sqrt{c^4/\omega^2+\delta^2}}{\delta}
 \, |o_1^{\check{\sigma}-\ddot{\sigma}}|.
\end{aligned}
\end{equation}
for $s\geq\check{s}$, $\ddot{\sigma}\geq\check{\sigma}$.
The asymptotics of the amplification factor $\bar{C}_{\vartheta\Theta}$ point to an even two derivative gain as compared to JMGT.
However, note that in the wave and JMGT case the poles ${\pole_\ell}$ lie on (or close to) the imaginary axis, whereas in the Westervelt case they are close to the negative real axis. This is in fact the more substantial contribution to ill-posedness of the inverse problem, in view of the fact that the observations are taken on the imaginary axis $\widetilde{\uld{p}}^{obs}(\imath m\omega)=\uld{p}_m^{obs}$.
Moreover a choice $s<\check{s}$ does not improve the asympotics of $\check{C}^{(\check{s}-s,\check{\sigma})}_{\vartheta\Theta}$, cf. \eqref{min_denominator}, and therefore would require square summability of the sequence $(\mathfrak{B}_m)_{m\in\mathbb{N}}$; it is not at all obvious whether and how this can be obtained, cf. Section~\ref{rem:B_m}. 

\subsubsection{Choice of source sequence $(\mathfrak{B}_m)_{m\in\mathbb{N}}$} \label{rem:B_m}
We now discuss conditions required on $(\mathfrak{B}_m)_{m\in\mathbb{N}}$, which via its definition as 
\[
\mathfrak{B}_m=\mathcal{B}_m(\hat{\psi},\hat{\psi}), \quad u^0_m(x)=\phi(x)\psi_m
\]
are in fact conditions on the choice of the reference pressure field $\uvec^0$.

The most crucial (and only indispensible for our uniqueness and stability result Theorem~\ref{thm:uniqueness-stability}) condition on $(\mathfrak{B}_m)_{m\in\mathbb{N}}$ is finiteness of the factor $\check{C}^{(\check{s}-s,\check{\sigma})}_{\vartheta\Theta}$ in \eqref{constants}, that is, summability of the expression defining $\check{C}^{(\check{s}-s,\check{\sigma})}_{\vartheta\Theta}$.
Due to \eqref{constants_JMGT} if $\tau>0$, this can be supported (for given, at most polynomially increasing $(\mathfrak{B}_m)_{m\in\mathbb{N}}$) by choosing the function spaces in $\XXX$ such that $\alpha=s-\check{s}>0$ is large enough.

While this first condition demands that the factors $\mathfrak{B}_m$ do not become to large, on the other hand the role of  $(|\widetilde{\mathfrak{B}}({\pole_j})|)_{j\in\mathbb{N}}$ in the definition of the image space norms \eqref{Ynorms_pol} favors smallness of their reciprocals.
Indeed, the relation \eqref{amplfac_lambda} that allows to control $|\cdot|_{\YYY^{mod}_{pol}}$ by a Sobolev norm, means that the asymptotic behaviour of the sequence 
$(|\widetilde{\mathfrak{B}}(\pole_j)|)_{j\in\mathbb{N}}$ 
must dominate the exponential expression
$e^{-\Re({\pole_j})T}$ up to a factor that is polynomial in $\lambda_j$.

These competing requirements on $\mathfrak{B}$ are cast into the two conditions in \eqref{kappas}, where $\kappa_1$, $\kappa_2$ have to satisfy \eqref{necessarykappa}.

A formal choice accommodating both (independently of the model $\vartheta$, $\Theta$), is 
defined by a delta pulse located at the end of the time interval $\psi_*^2=\delta_T$, which yields
$\widetilde{\mathfrak{B}_*}(o)=\revision{\frac{2}{T}}\langle \psi_*^2, e^{-o\cdot}\rangle_{\mathfrak{M}(0,T),C(0,T)}
= \revision{\frac{2}{T}}\, e^{-oT}$ and thus
\begin{equation}\label{BmdeltaT}
|\widetilde{\mathfrak{B}_*}({\pole_j})| = \revision{\tfrac{2}{T}} e^{-\Re({\pole_j}) T}, \quad 
|\mathfrak{B}_{*m}| = \revision{\tfrac{2}{T}}|\, e^{-\imath m\omega T}| = \revision{\tfrac{2}{T}},
\end{equation}
therefore implying 
\eqref{kappas} with 
\[
\kappa_1=
\liminf_{j\to\infty}
\begin{cases}
\tfrac{\ln(-\Re({\pole_j}))}{2\ln(\lambda_j)}&\text{ if }\ddot{\sigma}=0\\
\tfrac{\ln(|{\pole_j}|)}{\ln(\lambda_j)}&\text{ if }\ddot{\sigma}=1 
\end{cases},
\qquad
\kappa_2=0.
\]
With \eqref{factors_JMGT}, \eqref{factors_West} we have
\begin{equation}\label{kappasJMGTWest}
\kappa_1=
\begin{cases}
\tfrac18&\text{ if }\tau>0, \ \delta>0, \ \ddot{\sigma}=0\\
\tfrac12&\text{ if }\tau>0, \ \phantom{\delta>0,} \ \ddot{\sigma}=1\\
\tfrac12&\text{ if }\tau=0, \ \delta>0, \ \ddot{\sigma}=0\\
1&\text{ if }\tau=0, \ \delta>0, \ \ddot{\sigma}=1
\end{cases},
\qquad
\kappa_2=0.
\end{equation}
Since $u^0$ must be chosen as an element of $H^{\check{\sigma}}(0,T;H^{\check{s}}(\Omega))$, we approximate $\psi_*$ by $\psi\in H^{\check{\sigma}}(0,T)$ such that 
$\|\psi^2-\delta_T\|_{\mathfrak{M}}<\epsilon$ and thus for any $f\in C(0,T)$ we have 
$|\int_0^T\psi^2(t)f(t)\, dt-f(T)|<\epsilon\sup_{t\in[0,T]}|f(t)|$. 
Applying this to $f(t)=\revision{\frac{2}{T}}e^{-ot}$ yields 
\[
|\revision{\tfrac{2}{T}}\int_0^T\psi^2(t)e^{-ot}\, dt-\revision{\tfrac{2}{T}}e^{-oT}|
\leq\tfrac{\epsilon}{T}\sup_{t\in[0,T]}|e^{-oT}|
=\tfrac{\epsilon}{T}\,e^{-\Re(o)T}
\]
for any $o\in\mathbb{C}$ with $\Re(o)\leq0$, thus 
\[
|\widetilde{\mathfrak{B}}({\pole_j})-\revision{\tfrac{2}{T}} e^{-{\pole_j} T}| 
\leq\tfrac{\epsilon}{T} e^{-\Re({\pole_j}) T}, \quad 
|\mathfrak{B}_m-\revision{\tfrac{2}{T}} e^{-\imath m\omega T}|\leq \tfrac{\epsilon}{T}.
\]
Therefore setting $\epsilon=\frac12$ and using the (inverse) triangle inequality suffices to obtain the estimates
\begin{equation}\label{Bmpsi}
|\widetilde{\mathfrak{B}}({\pole_j})| \geq\tfrac{1}{2T} e^{-\Re({\pole_j}) T}, \quad 
|\mathfrak{B}_m| \leq \tfrac{3}{2T},
\end{equation}
and preserve \eqref{kappasJMGTWest}.

\section{Newton's method for the inverse problem }\label{sec:Newton}

Here, like in \cite{nonlinearity_imaging_2d}, we consider 
\[
\vec{F}:Q\times V^{\mathbb{N}}\to W^{\mathbb{N}}\times Y^{\mathbb{N}}
\] 
on the function spaces
\begin{equation}\label{eqn:DBQVW}
\begin{aligned}
&Q = L^2(\Omega), \quad V= H^2(\Omega), \quad W = L^2(\Omega), \quad Y = L^2(\Sigma), 
\end{aligned}
\end{equation}
Note that we have considered $\vec{F}$ on different function spaces in Section~\ref{sec:uniqueness}.

Analogously to \cite{nonlinearity_imaging_2d} a linearized range invariance condition can be established by extending the dependency of $\nlc$ 
(that is, introducing an artificial dependency of $\nlc$ on $m$, \revision{lifting it to a higher dimensional space})
to $\vec{\nlc}=(\nlc_m)_{m\in\mathbb{N}}$ in a neighborhood $U$ of $(0,\uvec_0)$
\begin{equation}\label{rangeinvar_diff}
\textup{ for all } (\vec{\nlc},\uvec)\in U \, \exists r(\vec{\nlc},\uvec)\in Q^\mathbb{N}\times V^\mathbb{N}\,: \  \vec{F}(\vec{\nlc},\uvec)-\vec{F}(0,\uvec_0)=\vec{F}'(0,\uvec_0)r(\vec{\nlc},\uvec),
\end{equation}
with $\uvec_0$ as in Theorem~\ref{thm:uniqueness-stability_lin}, cf. \eqref{u0separable} 
\begin{equation}\label{rid_cr}
\begin{aligned}
\exists\, c_r\in(0,1)\, \forall (\vec{\nlc},\uvec)\in U(\subseteq {Q^\mathbb{N}\times V^\mathbb{N}})\, : \;&
\|r(\vec{\nlc},\uvec)-r(\vec{\nlc}^\dagger,\uvec^\dagger)-(\vec{\nlc}-\vec{\nlc}^\dagger,\uvec-\uvec^\dagger)\|_{Q^\mathbb{N}\times V^\mathbb{N}}\\
&\leq c_r\|(\vec{\nlc}-\vec{\nlc}^\dagger,\uvec-\uvec^\dagger)\|_{Q^\mathbb{N}\times V^\mathbb{N}}.
\end{aligned}
\end{equation}

Together with the linearized uniqueness result in Section~\ref{sec:uniqueness} this implies convergence of a frozen Newton method
\begin{equation}\label{frozenNewtonHilbert}
x_{n+1}^\varepsilon \in \mbox{argmin}_{x\in U}
\|\vec{F}'(x_0)(x-x_n^\varepsilon)+\vec{F}(x_n^\varepsilon)-y^\varepsilon\|_Y^2+\alpha_n\|\vec{\nlc}-\vec{\nlc}_0\|_{Q^M}^2+\|P\vec{\nlc}\|_{Q^M}^2.
\end{equation}
where $x=(\vec{\nlc},\uvec)$, $x_0=(0,\uvec_0)$, and $y^\varepsilon$ is the actually given noisy data.  
Here we penalize the artificial dependence of $\vec{\nlc}$ on $m$  by a term $P\vec{\nlc}\in Q^\mathbb{N}$
\[
(P\vec{\eta})_m = \nlc_m - \frac{\sum_{n=1}^\infty n^{-2}\,\nlc_n}{\sum_{n=1}^\infty n^{-2}}, 
\]
where the weights $n^{-2}$ in the $\ell^2$ projection are introduced in order to enforce convergence of the sums while allowing for the desired $m$-independent setting $\nlc_m\equiv\nlc$.

The convergence proof can be literally taken from \cite[Theorem 3.3]{nonlinearity_imaging_2d} by just adding the $\tau u_{ttt}$ term.
\begin{theorem}\label{thm:convfrozenNewton}
Let  $x^\dagger=(\vec{\nlc}^\dagger,\uvec^\dagger)$ be a solution to $\vec{F}(x)=y^\varepsilon$ and let for the noise level $\varepsilon\geq\|y^\varepsilon-y\|_Y$ the stopping index $n_*=n_*(\varepsilon)$ be chosen such that 
\begin{equation}\label{nstar}
n_*(\varepsilon)\to0, \quad \varepsilon\sum_{j=0}^{n_*(\varepsilon)-1}c_r^j\alpha_{n_*(\varepsilon)-j-1}^{-1/2} \to 0 \qquad \textup{ as }\varepsilon\to0
\end{equation}
with $c_r$ as in \eqref{rid_cr}.
Moreover, let the assumptions of Theorem~\ref{thm:uniqueness-stability_lin} be satisfied.

Then there exists $\rho>0$ sufficiently small such that for $x_0\in\mathcal{B}_\rho(x^\dagger)\subseteq U$ the iterates $(x_n^\varepsilon)_{n\in\{1,\ldots,n_*(\varepsilon)\}}$ are well-defined by \eqref{frozenNewtonHilbert}, remain in $\mathcal{B}_\rho(x^\dagger)$ and converge in $Q^{\mathbb{N}}\times V^{\mathbb{N}}$, $\|x_{n_*(\varepsilon)}^\varepsilon-x^\dagger\|_{Q^{\mathbb{N}}\times V^{\mathbb{N}}}\to0$ as $\varepsilon\to0$. In the noise free case $\varepsilon=0$, $n_*(\varepsilon)=\infty$ we have $\|x_n-x^\dagger\|_{Q^{\mathbb{N}}\times V^{\mathbb{N}}}\to0$ as $n\to\infty$.
\end{theorem}

\section*{Acknowledgment}
This research was funded in part by the Austrian Science Fund (FWF) 
[10.55776/P36318]. 
\revision{The author wishes to thank the reviewers for their careful reading of the manuscript and their detailed reports with valuable comments and suggestions that have led to an improved version of the paper.
}

\end{document}